\documentclass[12pt]{amsart}
\usepackage{amsmath,amssymb,amsthm,amscd}
\usepackage[latin1]{inputenc}
\usepackage{wasysym}
\usepackage[all]{xy}
\usepackage[bf,small]{caption}
\usepackage{subfigure}
\usepackage{cite}
\usepackage{paralist}
\usepackage{graphicx}
\usepackage[usenames,dvipsnames]{color}
\usepackage{eurosym}
\usepackage{setspace}
\usepackage{dsfont}
\usepackage{fancyhdr}
\usepackage{remreset}
\usepackage[inner=2.5cm,outer=2.5cm,top=2cm,bottom=2cm,
includeheadfoot]{geometry}

\newtheorem{theorem}{Theorem}[section]
\newtheorem{cor}[theorem]{Corollary}
\newtheorem{lemma}[theorem]{Lemma}
\newtheorem{prop}[theorem]{Proposition}
\newtheorem{defi}[theorem]{Definition}
\newtheorem*{rem}{Remark}
\newtheorem*{theo}{Theorem}

\DeclareMathOperator \add{add}
\DeclareMathOperator \ann{ann}
\DeclareMathOperator \coker{coker}
\DeclareMathOperator \im{im}
\DeclareMathOperator \CR{CR}

\DeclareMathOperator \Ext{Ext}
\DeclareMathOperator \Jt{JType}
\DeclareMathOperator \modi{mod}

\DeclareMathOperator \pd{pd}

\DeclareMathOperator \Hom{Hom}
\DeclareMathOperator \End{End}
\DeclareMathOperator \EIP{EIP}
\DeclareMathOperator \EKP{EKP}
\DeclareMathOperator \CJT{CJT}

\DeclareMathOperator \rad{rad}
\DeclareMathOperator \ind{ind}

\DeclareMathOperator \supp{supp}
\DeclareMathOperator \soc{soc}

\DeclareMathOperator \GL{GL}

\DeclareMathOperator \rk{rk}

\begin{document}

\title{Categories of modules for elementary abelian $p$-groups and generalized Beilinson algebras}
\author{Julia Worch}
\address{Christian-Albrechts-Universit\"at zu Kiel, Ludewig-Meyn-Str. 4, 24098 Kiel, Germany}
\email{worch@math.uni-kiel.de}

\thanks{Partly supported by the D.F.G. priority programm SPP 1388  ``Darstellungstheorie''}

\begin{abstract}
In this paper, we approach the study of modules of constant Jordan type and equal images modules over elementary abelian $p$-groups $E_r$ of rank $r\geq 2$ by exploiting a functor from the module category of a generalized Beilinson algebra $B(n,r)$, $n \leq p$, to $\modi E_r$. \\
We define analogs of the above mentioned properties in $\modi B(n,r)$ and give a homological characterization of the resulting subcategories via a $\mathbb{P}^{r-1}$-family of $B(n,r)$-modules of projective dimension one. This enables us to apply general methods from Auslander-Reiten theory and thereby arrive at results that, in particular, contrast the findings for equal images modules of Loewy length two over $E_2$ \cite{cfs09} with the case $r > 2$.
Moreover, we give a generalization of the $W$-modules defined by Carlson, Friedlander and Suslin in \cite{cfs09}.
\end{abstract}

\maketitle

\section*{Introduction}

Addressing representations of finite group schemes over fields of positive characteristic, Carlson, Friedlander and Pevtsova have introduced in \cite{cafrpe08} the category of modules of constant Jordan type. Their approach involves the theory of $\pi$-points, i.e. certain embeddings $\alpha: k[T]/(T^p) \rightarrow kG$ along which representations of $kG$ can be restricted to the less complicated subalgebra $k \mathbb{Z}_p \cong k[T]/(T^p)$. The representations of $k[T]/(T^p)$ are completely understood in terms of Jordan types, i.e. Jordan block decompositions. Since $kG$ is wild in most cases, it is reasonable to study representations with additional properties. A $kG$-module has constant Jordan type if its Jordan block decomposition does not depend on the chosen $\pi$-point. There is the related notion of the constant $j$-rank property such that a module has constant Jordan type iff it has constant $j$-rank for all $j \geq 1$ (cf. \cite[p. 11]{fp}).\\

Confining investigations to elementary abelian $p$-groups $E_r=(\mathbb{Z}_p)^{\times r}$ of rank $r \geq 2$, a more restrictive condition has been formulated in \cite{cfs09} by Carlson, Friedlander and Suslin, where $M \in \modi kE_r$ satisfies the so-called equal images property if there exists a $k$-space $V$ such that $\alpha(t).M=V$ for all $\pi$-points $\alpha$. The dual concept is referred to as the equal kernels property. In \cite{cfs09}, the authors are mainly concerned with the case $r=2$ and they introduce a family of $kE_2$-modules, the so-called $W$-modules, which satisfy the equal images property and are ubiquitous in a sense that every module satisfying the equal images property is a quotient of a $W$-module \cite[4.4]{cfs09}. This relies on the fact that the indecomposable equal images modules of Loewy length two over $kE_2$ are $W$-modules \cite[4.1]{cfs09}. It has been observed that these modules correspond to the preinjective modules over the Kronecker algebra (cf. \cite[4.2.2]{f11}).\\

The approach we give in this paper is based on the objective to understand $\modi_n kE_r$, i.e. the full subcategory of $kE_r$-modules with Loewy length bounded by $n \leq p$. The generalized Beilinson algebra $B(n,r)$ provides a faithful exact functor $\mathfrak{F}: \modi B(n,r) \rightarrow \modi_n(kE_r)$. We formulate analogs of the constant Jordan type and constant $j$-rank property as well as the equal images and equal kernels property for $ B(n,r)$-modules and define full subcategories $\CJT(n,r),\ \CR^j(n,r), \ \EIP(n,r), \ \EKP(n,r)\subset \modi B(n,r)$ such that the restrictions of $\mathfrak{F}$ to $\EIP(n,r)$ and $\EKP(n,r)$ reflect isomorphisms and have an essential image consisting of standardly graded modules with the equal images property and costandardly graded modules with the equal kernels property, respectively.\\

An immediate advantage of passing over to $B(n,r)$ is that we are able to give a homological characterization of the categories $\EIP(n,r)$ and $\EKP(n,r)$ involving a $\mathbb{P}^{r-1}$-family of $B(n,r)$-modules of projective dimension one which allows us to apply general methods from Auslander-Reiten theory. With this tool in hand, we prove:
\begin{theo}[A]
The category $\EIP(n,r)$ is the torsion class $\mathcal{T}$ of a torsion pair $(\mathcal{T}, \mathcal{F})$ in $\modi B(n,r)$ such that $\EKP(n,r) \subset \mathcal{F}$ and $\mathcal{T}$ is closed under the Auslander-Reiten translate $\tau$ and contains all preinjective modules.\\
Dually, $\EKP(n,r)$ is the torsion-free class $\mathcal{F}'$ of a torsion pair $(\mathcal{T'}, \mathcal{F'})$ in $\modi B(n,r)$  such that $\EIP(n,r) \subset \mathcal{T}'$ and $\mathcal{F}'$ is closed under $\tau^{-1}$ and contains all preprojective modules.\\
In particular, there are no non-trivial maps $\EIP(n,r) \rightarrow \EKP(n,r)$.
\end{theo}

We can specialize our results to the case $n=2$: The algebra $B(2,r)$ is the path algebra $\mathcal{K}_r$ of the $r$-Kronecker and has wild representation type if and only if $r > 2$. The Auslander-Reiten quiver of the wild hereditary algebra $B(2,r)$ consists of a preprojective component, a preinjective component and infinitely many (regular) components of type $\mathbb{Z}A_{\infty}$ \cite{ri78}. We summarize our main results for $B(2,r), \ r >2,$ as follows, contrasting the findings for $r=2$:
\begin{theo}[B] Let $r >2$, $n \leq p$ and let $\Gamma$ be the Auslander-Reiten quiver of $B(2,r)$.
\begin{enumerate}[(i)]
\item Let $\mathcal{C}$ be a regular component of $\Gamma$. Then $\EIP(2,r) \cap \mathcal{C}$ and $\EKP(2,r) \cap \mathcal{C}$ are non-empty disjoint cones. The size of the gap $\mathcal{W}(\mathcal{C}) \in \mathbb{N}_0$ between these cones is an invariant of $\mathcal{C}$.
\item For each $n \in \mathbb{N}$, there exists a regular component $\mathcal{C}$ of $\Gamma$ such that $\mathcal{W}(\mathcal{C})>n$.
\item If $\mathcal{W}(\mathcal{C})=0$, then every object in $\mathcal{C}$ has constant Jordan type.
\item  If $\mathcal{W}(\mathcal{C})=1$, then either $\mathcal{C} \subset \CJT(2,r)$ or apart from the cones $\EIP(2,r) \cap \mathcal{C}$ and $\EKP(2,r) \cap \mathcal{C}$, there are no other objects of constant Jordan type in $\mathcal{C}$.
\end{enumerate}
\end{theo}
 
Our paper is organized as follows: In Section 1, we recall definitions and basic results and give a generalization of the $W$-modules defined in \cite{cfs09} to arbitrary rank. We introduce generalized Beilinson algebras and give a homological description of the categories $\CJT(n,r),\ \CR^j(n,r), \ \EIP(n,r)$ and $\EKP(n,r)$ in Section 2 and point out the special role that generalized $W$-modules play in $\EIP(n,r)$. In the final section, we restrict our investigations to modules of Loewy length two and give our more specific results on the $r$-Kronecker together with some examples.

\section{Generalized $W$-modules}
Let us first of all introduce the set up and recall the relevant concepts and some basic results from \cite{cafrpe08}, \cite{cafr09}, \cite{cfs09} and \cite{fp}. In doing so, we will present some definitions in a way that is suitable for our purposes.\\

Let $k$ be an algebraically closed field of characteristic $p>0$. Let $E_r=(\mathbb{Z}_p)^{\times r}$ be an elementary abelian $p$-group of rank $r\geq2$ with generators $g_1,\ldots,g_r$. Let furthermore $R=k[X_1,\ldots,X_r]$ be the polynomial ring in $r$ variables. Sending $X_i$ to $x_i:=g_i-1$ yields an isomorphism $k[X_1,\ldots,X_r]/(X_1^p,\ldots,X_r^p) \cong kE_r$ between the truncated polynomial ring and the group algebra of $E_r$. Consider furthermore the ideal $I=(X_1,\ldots,X_r)\subseteq R$ generated by all polynomials of degree one as well as the augmentation ideal $J=\rad(kE_r)=(x_1,\ldots,x_r)$ of $kE_r$. We let $\modi kE_r$ be the category of finitely generated $kE_r$-modules and $\modi_n (kE_r) \subset \modi kE_r$ be the full subcategory consisting of modules of  Loewy length at most $n$.\\

An algebra homomorphism $\alpha: k[T]/(T^p) \rightarrow kE_r$ is called \textit{p-point} if the pullback $\alpha^*(kE_r)$ is a free $k[T]/(T^p)$-module. Note that this is equivalent to saying that $\alpha(t)$ with $t:=T+(T^p)$ is an element in $\rad(kE_r) \backslash \rad^2(kE_r)$ \cite[p. 3]{cfs09}. Given such a $p$-point $\alpha$, for $M \in \modi kE_r$, we consider the linear operator $\alpha(t)_M : M \rightarrow M, \ m \mapsto \alpha(t).m$. The Jordan canonical form of $\alpha(t)_M$ entirely determines the isomorphism type of $M$ as a $k[T]/(T^p)$-module. The sequence of sizes of Jordan blocks is referred to as the {\it Jordan type} of $M$ corresponding to $\alpha$ and we write $\Jt(\alpha,M)=a_p[p] +\cdots + a_1[1]$, indicating that there are $a_i$ blocks of size $[i]$ for $1 \leq i \leq p$. If this Jordan type does not depend on the $p$-point we choose, we say that $M$ is of {\it constant Jordan type} $\Jt(M):=\Jt(\alpha,M)$.\\
 We say that $M \in \modi kE_r$ is of \textit{constant $j$-rank} for $j \in \mathbb{N}$ if the rank $\rk \alpha(t)^j_M$ is independent of our choice of $p$-point. Note that $M$ is of constant Jordan type iff $M$ is of constant $j$-rank for all $j \geq 1$ \cite[p. 11]{fp}. We denote the subcategories of $\modi kE_r$ consisting of such modules by $\CR^j(kE_r)$ and $\CJT(kE_r)$.\\
A module $M \in \modi kE_r$ is said to satisfy the \textit{equal images property} if $\im \alpha(t)_M=\rad M$ for all $p$-points $\alpha$. A module $M \in \modi kE_r$ is said to satisfy the \textit{equal kernels property} if $\ker \alpha(t)_M=\soc M$ for all $p$-points $\alpha$. We denote the corresponding subcategories of $\modi(kE_r)$ by $\EIP(kE_r)$ and $\EKP(kE_r)$, respectively.\\
In \cite[1.2, 1.7]{cfs09}, it is shown that it  suffices to check the above properties for all $p$-points $\alpha$ with $\alpha(t)=\alpha_1x_1+\cdots + \alpha_r x_r$ for a non-trivial element $(\alpha_1,\ldots,\alpha_r) \in k^r \backslash 0$.\\
Note that $M$ satisfies the equal images property if and only if its linear dual $M^*$ satisfies the equal kernels property. The category $\EIP(kE_r)$ is image-closed \cite[1.10]{cfs09}, and dually $\EKP(kE_r)$ is closed under taking submodules. We have $\EIP(kE_r) \cup \EKP(kE_r) \subseteq \CJT(kE_r)$ \cite[1.9]{cfs09} and furthermore $\EIP(kE_r) \cap \EKP(kE_r)=\add k$ \cite[4.4.3]{f11}, where $k$ is the trivial $kE_r$-module and $\add k$ the full subcategory of $\modi kE_r$ whose objects are direct sums of the trivial module $k$.\\

We now give a generalization of the $\mathbb{Z}_p \times \mathbb{Z}_p$-modules $W_{n,d}$ defined by Carlson, Friedlander and Suslin in \cite{cfs09} to elementary abelian $p$-groups of arbitrary rank. The authors show that these so-called $W$-modules are indecomposable equal images modules which play a prominent role in the category $\EIP(k(\mathbb{Z}_p \times \mathbb{Z}_p))$.
Whereas in \cite{cfs09}, the modules $W_{n,d}$ are defined via generators and relations, we give an alternative definition that is amenable to generalization to higher rank.\\

For all $n \in \mathbb{N},\ d \leq \min\left\{n,p\right\}$, we consider the $R$-module
$$M_{n,d}^{(r)}:=I^{n-d}/I^n.$$
By choice of $d$, the canonical action factors through $R/(X_1^p,\ldots,X_r^p)$, so that we can likewise study this module and its linear dual $W_{n,d}^{(r)}:=({M_{n,d}^{(r)}})^*$ over $kE_r$.\\
The module ${M_{3,2}^{(3)}}$ can be depicted as follows
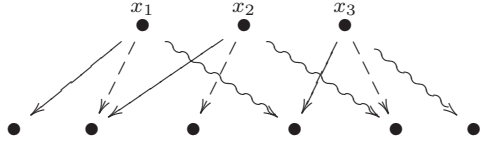
\[ \begin{xy}
\xymatrix@R=2em@C=0.4em{
& & & \stackrel{x_1}{\bullet} \ar@{->}[llld] \ar@{-->}[ld] \ar@{~>}[rrrd] &&  \stackrel{x_2}{\bullet} \ar@{->}[llld] \ar@{-->}[ld] \ar@{~>}[rrrd]&&  \stackrel{x_3}{\bullet} \ar@{~>}[rrrd] \ar@{-->}[rd] \ar@{->}[ld]\\
{\bullet} &&  {\bullet}  && \bullet  && \bullet  && \bullet  && \bullet}
\end{xy} \]
where the dots represent the canonical basis elements given by the monomials in degree one and two and $\rightarrow, \ \dashrightarrow$ and $\rightsquigarrow$ denote the action of $x_1, x_2$ and $x_3$, respectively.
It is easy to see that in case $r=2$ we have $W_{n,d}^{(2)}=W_{n,d}$ as defined in \cite{cfs09}.
Modules of the form $M_{n,d}^{(r)}$ will be referred to as {\it $M$-modules} and modules of the form $W_{n,d}^{(r)}$ as {\it $W$-modules}, respectively. We furthermore set the convention that $M_{n,d}^{(r)}:=M_{n,n}^{(r)}$ and  $W_{n,d}^{(r)}:= W_{n,n}^{(r)}$ in case $d>n$.\\
The module $M_{n,d}^{(r)}$ satisfies the equal kernels property, since for all $(\alpha_1,\ldots,\alpha_r) \in k^r \backslash 0$ we have
$$\ker \left\{\sum_{i=1}^r{\alpha_ix_i}: M_{n,d}^{(r)} \rightarrow M_{n,d}^{(r)} \right\}=I^{n-1}/I^n=\soc(M_{n,d}^{(r)}).$$
Hence $W_{n,d}^{(r)}$ satisfies the equal images property. Some $W$-modules can be recognized as submodules of the group algebra $kE_r$, generalizing \cite[2.2]{cfs09}.
\begin{prop}
There is an isomorphism
$$W_{d,d}^{(r)} \cong \rad^{r(p-1)+1-d}(kE_r)$$
for $d\leq p$.
\end{prop}
\begin{proof}
Observe  that $kE_r$ is isomorphic to the restricted enveloping algebra of an $r$-dimensional abelian Lie algebra with trivial $p$-map \cite[\S 5]{fs88} and is equipped with the structure of a Frobenius algebra where the projection $\tau: kE_r \rightarrow k$ onto the coefficient of $x_1^{p-1} \cdots x_r^{p-1}$ defines a non-degenerate associative symmetric bilinear-form
$$(.,.): kE_r \times kE_r \rightarrow k, (a,b):=\tau(ab),$$
see \cite{ber64}. Since there is an isomorphism $kE_r/\rad^d(kE_r) \cong M^{(r)}_{d,d}$, the claimed isomorphism of $kE_r$-modules follows from the associativity of $(.,.)$ together with the isomorphism
$$W_{d,d}^{(r)} \cong (kE_r/\rad^d(kE_r))^* \cong (\rad^d(kE_r))^{\perp} =\rad^{r(p-1)+1-d}(kE_r).$$
\end{proof}

Observe furthermore that the algebraic group $\GL_r(k)$ acts on the $r$-dimensional vector space $\bigoplus_{i=1}^r kX_i$ and thereby on $R$ and $kE_r$ via automorphisms, leaving $I$ and $J$ invariant. Moreover, consider the action of $\GL_r(k)$ on $\modi kE_r$ sending $M \in \modi kE_r$ to its {\it $g$-twist $M^{(g)}$} for $g \in \GL_r(k)$, where $M^{(g)}$ is the $kE_r$-module with underlying vector space $M$ and action given by $x.m:=(g^{-1}.x)m$. We call a module {\it $\GL_r(k)$-stable} if there is an isomorphism $M \cong M^{(g)}$ for all $g \in \GL_r(k)$.  Since $\GL_r(k)$ acts on $\bigoplus_{i=1}^r kX_i \backslash 0$ with one orbit, $\GL_r(k)$-stable modules are necessarily of constant Jordan type.
\begin{prop}\label{stable}
Let $n \in \mathbb{N}, d \leq p$. The $kE_r$-modules $M_{n,d}^{(r)}$ and $W_{n,d}^{(r)}$ are $\GL_r(k)$-stable.
\end{prop}
\begin{proof}
Since dualizing and twisting are compatible, it suffices to prove the first claim: The module $I^{n-d}/I^n$ is a subfactor of the $\GL_r(k)$-module $R$ and the map
$$\varphi: I^{n-d}/I^n \rightarrow I^{n-d}/I^n, m \mapsto g^{-1}.m$$ 
defines an isomorphism $M_{n,d}^{(r)} \cong (M_{n,d}^{(r)})^{(g)}$.
\end{proof}
\vspace{0,5 cm}
In the following, we will make use of the graded structures of the algebras $R$ and $kE_r$ and their module categories, respectively.\\
Graded (Artin) algebras and their modules categories were studied by Gordon and Green \cite{gg82}, \cite{gg82b}. An algebra $\Lambda$ is {\it $\mathbb{Z}^n$-graded} for some $n \in \mathbb{N}$ if $\Lambda$ affords a vector space decomposition $\Lambda=\bigoplus_{i \in \mathbb{Z}^n} \Lambda_i$ such that $\Lambda_i \Lambda_j \subseteq \Lambda_{i+j}$ for all $i,j \in \mathbb{Z}^n$. We denote by $|i|:=\sum_{j=1}^n i_j$ the {\it value} of $i=(i_1,\ldots,i_n) \in \mathbb{Z}^n$. Graded ideals are defined canonically. A $\Lambda$-module $M$ is {\it $\mathbb{Z}^n$-graded} if $M=\bigoplus_{i \in \mathbb{Z}^n} M_i$ such that $\Lambda_i M_j \subseteq M_{i+j}$ for all $i,j \in \mathbb{Z}^n$.\\
The category $\modi_{\mathbb{Z}^n}\Lambda$ has the finitely generated $\mathbb{Z}^n$-graded $\Lambda$-modules as objects and the sets of morphisms $\Hom_{\Lambda}^{\mathbb{Z}^n}(M,N)$ are the $\Lambda$-linear maps $\varphi: M \rightarrow N$ with $\varphi(M_i) \subseteq N_i$ for all $i \in \mathbb{Z}^n$. Furthermore, the {\it $i$-th shift functor} $[i]: \modi_{\mathbb{Z}^n}\Lambda \rightarrow \modi_{\mathbb{Z}^n}\Lambda$ is defined on objects $M \in \modi_{\mathbb{Z}^n}\Lambda$ to be $M[i]$ where $M[i]_j:=M_{j-i}$. Morphisms are left unchanged.\\
If $M, N \in \modi_{\mathbb{Z}^n}\Lambda$ afford gradings $M=\bigoplus_{i \in \mathbb{Z}^n} M_i$ and $N=\bigoplus_{i \in \mathbb{Z}^n} N_i$, then
$\Hom_{\Lambda}(M,N)$ affords a grading $\Hom_{\Lambda}(M,N)=\bigoplus_{i \in \mathbb{Z}^n}\Hom_{\Lambda}(M,N)_i$ as a module over the $\mathbb{Z}^n$-graded algebra $\End_{\Lambda}(N)$, where $\Hom_{\Lambda}(M,N)_i=\left\{ \varphi \in \Hom_{\Lambda}(M,N) | \varphi(M_j) \subseteq N_{i+j} \; \forall j \in \mathbb{Z}^n \right\}$.\\

Now $R=\oplus_{i \in \mathbb{Z}^r} R_i$ is a $\mathbb{Z}^r$-graded algebra, where $R_i$ is the $k$-span of the polynomial $X_1^{i_1} \cdots X_r^{i_r}$   for all $i=(i_1,\ldots,i_r) \in \mathbb{N}_0^r$ and $R_i=0$ else. Hence all non-trivial homogeneous components are one-dimensional.\\
Since the ideal $(X_1^p,\ldots,X_r^p)$ is homogeneous with respect to this grading, $kE_r$ inherits the  $\mathbb{Z}^r$-grading from $R$. Furthermore it is $I=\bigoplus_{i \in \mathbb{Z}^r \atop i \neq 0} R_i$ and hence
 $M_{n,d}^{(r)}=I^{n-d}/I^n$ has both as an $R$- and $kE_r$-module a canonical $\mathbb{Z}^r$-grading 
$$M_{n,d}^{(r)}=\bigoplus_{i \in \mathbb{Z}^r} M_i$$
where $M_i$ is the vector space spanned by $x_1^{i_1} \cdots x_r^{i_r}=X_1^{i_1} \cdots X_r^{i_r} + I^n$ for all $i \in \mathbb{N}_0^r$ with $n-d \leq |i| \leq n-1$, and $M_i=0$ else. Endowed with this grading, $M_{n,d}^{(r)}$ is generated by its components $M_i$ with $i \in \mathbb{N}_0^r,\ |i|=n-d$.
Observe that the $\mathbb{Z}^r$-grading induces a $\mathbb{Z}$-grading both on the algebra and the graded modules in a canonical fashion via $R_i=\bigoplus_{|(j_1,\ldots,j_r)|=i}R_{(j_1,\ldots,j_r)}$ and $M_i=\bigoplus_{|(j_1,\ldots,j_r)|=i}M_{(j_1,\ldots,j_r)}$ for all $i \in \mathbb{Z}$.

\begin{theorem}
For $r\geq2$ and $n \geq d>1$, $d \leq p$, we have an isomorphism of $\mathbb{Z}^r$-graded rings
$$\End_{kE_r}(M_{n,d}^{(r)}) \cong kE_r/J^d \oplus \bigoplus_{i \in \mathbb{Z}^r \atop |i|=d-1}k[i]^{s_i}$$
where the right-hand side denotes the trivial extension of $kE_r/J^d$ by a sum of shifts of the trivial $kE_r$-bimodule $k$.
In particular, $\End_{kE_r}(M_{n,d}^{(r)})$ is local and commutative.
\end{theorem}

\begin{rem}
\textnormal{By computing $\Hom$-spaces, we will moreover show that the $s_i$ are uniquely determined and that they are all equal to zero iff $n=d$.}
\end{rem}
\begin{proof} 
We claim that for $n \in \mathbb{N}, d \leq \min\left\{n, p \right\}$, there is a monomorphism 
$$\iota: kE_r/J^d \rightarrow \End_{kE_r}(M_{n,d}^{(r)})$$
of $\mathbb{Z}^r$-graded $k$-algebras. Multiplication by an element of $kE_r$ clearly yields an endomorphism of $M_{n,d}^{(r)}$ and we obtain a homomorphism 
$kE_r \rightarrow \End_{kE_r}(M_{n,d}^{(r)})$
of $k$-algebras which obviously respects the $\mathbb{Z}^r$-grading. Since $\ann_{kE_r}(M_{n,d}^{(r)})=J^d$, $\iota$ is injective.\\
We now show that $\iota$ induces an isomorphism of homogeneous components
\begin{equation}\label{hc}
 ( kE_r/J^d)_i \cong \End(M_{n,d}^{(r)})_i
\end{equation}
for $i \in \mathbb{Z}^r$ with $|i| \leq d-2$. \\
\textit{Proof of (\ref{hc}):}
Since $M_{d,d}^{(r)} \cong  kE_r/J^d$, the isomorphism in (\ref{hc}) is obvious for $n=d$. We thus assume $n >d$.
Let $\varphi_i \in \End(M_{n,d}^{(r)})_i$ and $|i| \leq d-2$. Recall that all non-trivial homogeneous components of $M_{n,d}^{(r)}$ are one-dimensional and the module is generated by its homogeneous components $M_j=\left< x_1^{j_1}\cdots x_r^{j_r} \right>_k$ with $|j|=n-d$.\\
For all $1 \leq t \leq r$, we denote by $1_t$ the element in $\mathbb{N}_0^r$ with the $t$-th entry being equal to 1 and all others being equal to 0. For all $1 \leq t, t' \leq r$ we denote by $-\mathds{1}_t+ \mathds{1}_{t'}$ the operation on $\mathcal{K}=\left\{ \kappa \in \mathbb{N}_0^r |  \ |\kappa|=n-d \right\}$ given by $-\mathds{1}_t+ \mathds{1}_{t'}(\kappa)=\kappa-1_t+1_{t'}$ if $\kappa_t \neq 0$ and $-\mathds{1}_t+ \mathds{1}_{t'}(\kappa)=\kappa$ else. Observe that every non-empty subset of $\mathcal{K}$ that is closed under all such operations is equal to $\mathcal{K}$.\\
Let us first of all show that $\varphi_i=0$ if $i$ has a negative entry $i_l$ for some $1 \leq l \leq r$. We know that $\varphi_i$ certainly vanishes on those $M_j$, $|j|=n-d$, with $j_l=0$.\\
Now assume $\varphi_i(M_k)=0$, i.e. $\varphi_i(x_1^{k_1}\cdots x_r^{k_r})=0$ for some $k \in \mathbb{N}_0^r,\ |k|=n-d$. Let furthermore $t \in \left\{1,\ldots,r\right\}$ such that $k_t  \neq 0$. For all $t' \in \left\{1,\ldots,r \right\}$, we have
\begin{equation}\label{cv}
\ x_t \varphi_i(x_1^{k_1}\cdots x_r^{k_r}\frac{x_{t'}}{x_t})= x_{t'} \varphi_i (x_1^{k_1}\cdots x_r^{k_r}).
\end{equation}
By our assumption, we have $|i| \leq d-2$ which implies $\varphi_i(x_1^{k_1}\cdots x_r^{k_r}\frac{x_{t'}}{x_t})=0$ and hence $\varphi_i(M_{k-1_t+1_{t'}})=0$. Thus $\left\{\kappa \in \mathcal{K}|\varphi_i(M_{\kappa})=0\right\}$ is non-empty and closed under operations of the form $-\mathds{1}_t+ \mathds{1}_{t'}$ and hence equal to $\mathcal{K}$. We thus obtain $\varphi_i(M_{\kappa})=0$ for all $\kappa \in \mathcal{K}$ and hence $\varphi_i=0$.\\
For $i \in \mathbb{N}_0^r$, $|i| \leq d-2$, we use the fact that non-trivial homogeneous components are one-dimensional and obtain $\varphi_i(x_1^{k_1}\cdots x_r^{k_r})=c_k  x_1^{k_1+i_1}\ldots x_r^{k_r+i_r}$ for all $k \in \mathbb{N}_0^r, |k|=n-d$ and scalars $c_k$. Comparing coefficents in (\ref{cv}) yields that $\varphi_i$ is multiplication by an element of the form $cx_1^{i_1} \cdots x_r^{i_r}$ with $c \in k$. This proves our claim (\ref{hc}).\\

In case $i \in \mathbb{N}_0^r,\ |i|=d-1$, we have an isomorphism of vector spaces 
$$\End_{kE_r}(M_{n,d}^{(r)})_{i} \cong \bigoplus_{j \in \mathbb{N}_0^r \atop |j|=n-d} \Hom_k((M_{n,d}^{(r)})_j, (M_{n,d}^{(r)})_{i+j})$$
and hence $\dim_k \End_{kE_r}(M_{n,d}^{(r)})_{i}/\iota((kE_r/J^d)_i) = \dim_k I^{n-d} -1$. The right-hand term is equal to zero if and only if $n=d$. For $i \in \mathbb{Z}^r \backslash \mathbb{N}_0^r,\ |i|=d-1$, we have
$$\End_{kE_r}(M_{n,d}^{(r)})_{i} \cong \bigoplus_{j \in \mathbb{N}_0^r \atop {|j|=n-d \atop i+j \in \mathbb{N}_0^r}} \Hom_k((M_{n,d}^{(r)})_j,(M_{n,d}^{(r)})_{i+j})$$
with $( kE_r/J^d)_i=(0)$ and the right-hand term being equal to zero iff $n=d$.
 Since furthermore $\End(M_{n,d}^{(r)})_i=0$ for $i \in \mathbb{Z}^r, |i| \geq d$, and maps of degree $d-1$ vanish when composed with maps of degree greater than zero, we obtain the above structure of the endomorphism ring.
\end{proof}

\begin{cor}\label{brick} Let $n \in \mathbb{N}$ and $1<d\leq p$.
\begin{enumerate}[(i)]
\item The $kE_r$-module $M_{n,d}^{(r)}$ is indecomposable. 
\item We have $k \cong\End_{kE_r}(M_{n,d}^{(r)})_0=\End_{kE_r}^{\mathbb{Z}}(M_{n,d}^{(r)})$, 
i.e. $M_{n,d}^{(r)}$ is a brick in $\modi_{\mathbb{Z}} kE_r$.
\end{enumerate} 
\end{cor}

Jordan types that can be realized via indecomposable modules are of special interest. Counting polynomials, we obtain:
\begin{prop} For $n \in \mathbb{N}, d \leq \min\left\{n, p \right\}$, we have
$$\Jt(M_{n,d}^{(r)} )={{r+n-d-1}\choose{n-d}}[d]+\sum_{i=1}^{d-1}\binom{r+n-2-i}{n-i}[i]=\Jt(W_{n,d}^{(r)})$$
and in particular for $n=d \leq p$
$$\Jt(M_{n,n}^{(r)})=[n]+\sum_{i=1}^{n-1}{{r+n-i-2} \choose {n-i}}[i]=\Jt(W_{n,n}^{(r)}).$$
\end{prop}
The indecomposability of $W$- and $M$-modules of Loewy length greater than one over the algebra $kE_2$ follows directly from \cite[4.2]{cfs09}, according to which the Jordan type $\sum_{i=1}^p a_i [i]$ of a module with the equal images property is such that $a_{i-1} \neq 0$ whenever $a_i \neq 0$, $i \geq 2$. Taking into account that $\EIP(kE_r)$ and $\EKP(kE_r)$ are closed under direct summands and $\Jt(W_{n,d})=(n-d+1)[d] + \sum_{i=1}^{d-1}[i]$, these modules are hence indecomposable if $d \geq 2$. In case $r>2$, this conclusion does not seem to follow from the computation of Jordan types.\\

Moreover, for $r=2$, the indecomposable equal images modules of Loewy length 2 are just the modules $W_{n,2}$ \cite[4.1]{cfs09}. We will show in the following sections that in case $r>2$, the situation is completely different and there is no hope to parametrize the indecomposable equal images modules of Loewy length 2 in the same fashion. It seems that $W$-modules are thus not ``ubiquitous'' in $\EIP(kE_r)$ if $r>2$.

\makeatletter
\@removefromreset{theorem}{section}
\@addtoreset{theorem}{subsection}
\renewcommand{\thetheorem}{\arabic{section}.\arabic{subsection}.\arabic{theorem}}
\makeatother

\section{Equal images modules for generalized Beilinson algebras}
In order to understand the subcategories of $\modi  kE_r$ introduced in the previous section, we will now consider the category $\modi_{\mathbb{Z}} kE_r$ of $\mathbb{Z}$-graded modules over the $\mathbb{Z}$-graded algebra $kE_r$. When studying objects in $\modi_{\mathbb{Z}} kE_r$ that have a bounded support, the generalized Beilinson algebra $B(n,r)$ comes into play. It turns out that we can define certain analogs of our subcategories of $\modi_n kE_r$ as subcategories of $\modi B(n,r)$ which exhibit interesting properties and behave nicely when it comes to Auslander-Reiten theory. Our main results strongly depend on our homological characterization of these subcategories. 
\subsection{General approach}

We recall from \cite{gg82} that there is a faithful exact functor $$\mathfrak{F}: \modi_{\mathbb{Z}} kE_r \rightarrow \modi  kE_r$$
 referred to as the forgetful functor since it ``forgets'' the grading on objects. $\mathfrak{F}$ preserves indecomposability and has the property that the fibre of an indecomposable object in the essential image of $\mathfrak{F}$ consists of the shifts $M[i],\ i \in \mathbb{Z},$ of a certain indecomposable object $M=\bigoplus_{i \in \mathbb{Z}}M_i \in \modi_{\mathbb{Z}} kE_r$ \cite[4.1]{gg82}. Note furthermore that $\mathfrak{F}$ is not dense.\\
If $M=\bigoplus_{i \in \mathbb{Z}} M_i \in \modi_{\mathbb{Z}} kE_r$, then $\supp(M)=\left\{i \in \mathbb{Z}|M_i \neq 0 \right\}$ is called the {\it support} of $M$. An object in the essential image of $\mathfrak{F}$ is called \textit{gradable}. We say that $M \in \modi kE_r$ is \textit{$J$-gradable} for $J \subseteq \mathbb{Z}$ if $M \cong \mathfrak{F}(\bigoplus_{i \in \mathbb{Z}} M_i)$ for some $\bigoplus_{i \in \mathbb{Z}} M_i \in \modi_{\mathbb{Z}} kE_r$ such that $\supp(\bigoplus_{i \in \mathbb{Z}} M_i ) \subseteq J$.\\
Using the terminology of \cite{gmmz}, the positively graded algebra $\Lambda=kE_r$ is {\it standardly graded}, i.e. $\Lambda_0$ is a direct product of $k$, $\Lambda_i$ is finite dimensional for all $i \geq 0$ and $\Lambda_i\Lambda_j=\Lambda_{i+j}$ for all $i,j \geq 0$. We call $M=\bigoplus_{i \in \mathbb{Z}}M_i \in \modi_{\mathbb{Z}} kE_r$ \textit{standardly graded} if $M$ is generated by $M_i$, where $i = \min \supp(M)$. Dually, we say that $M$ is \textit{costandardly graded} if $M$ is cogenerated by $M_i$, where $i = \max \supp (M)$. We refer to their images under $\mathfrak{F}$ as {\it standardly gradable} and {\it costandardly gradable} objects, respectively.\\
 For $2\leq n \leq p$, we now consider the full subcategory $\mathcal{C}_{[0,n-1]}$ of $\modi_{\mathbb{Z}} kE_r$ containing those objects $M=\bigoplus_{i \in \mathbb{Z}}M_i \in \modi_{\mathbb{Z}} kE_r$ with $\supp(M) \subseteq [0,n-1]:=\left\{0,\ldots,n-1\right\}$. Hence the essential image of $\mathfrak{F}| _{\mathcal{C}_{[0,n-1]}}$ consists of the $[0,n-1]$-gradable objects in $\modi kE_r$. Observe furthermore that $\mathcal{C}_{[0,n-1]}$ is equivalent to $\modi B(n,r)$, the module category of a generalized Beilinson algebra, where $B(n,r)$ is defined as follows:\\
Let $E(n,r)$ be the path algebra of the quiver with $n$ vertices and $r$ arrows between vertices $i$ and $i+1$ for all $0 \leq i \leq n-1$.

$$\begin{xy}
\xymatrix{
0 \ar@/^1pc/[r]^{\gamma_1^{(0)}} \ar@/_1pc/[r]_ {\gamma_r^{(0)}} \ar@{}[r]|{\vdots} & 1 \ar@/^1pc/[r]^{\gamma_1^{(1)}} \ar@/_1pc/[r]_ {\gamma_r^{(1)}}  \ar@{}[r]|{\vdots} & 2 \ar@{}[r]|{\cdots} & n-2 \ar@/^1pc/[r]^{\gamma_{1}^{({n-2})}} \ar@/_1pc/[r]_ {\gamma_{r}^{({n-2})}}  \ar@{}[r]|{\vdots} & n-1 }
\end{xy}$$

Now let $B(n,r)$ be the factor algebra $E(n,r)/I$ where $I$ is generated by the commutativity relations $\gamma^{(j)}_i \gamma^{(j-1)}_{k}-\gamma^{(j)}_{k}  \gamma^{(j-1)}_{i}$ for all $i,k \in \left\{1,\ldots,r\right\}, j \in \left\{1,\ldots,n-2\right\}$.
The equivalence between $\mathcal{C}_{[0,n-1]}$ and $\modi B(n,r)$ is such that $M=M_0 \oplus \cdots \oplus M_{n-1} \in \mathcal{C}_{[0,n-1]}$ is a module for $B(n,r)$ where $M_i=e_i M$ for the primitive orthogonal idempotents $e_i \in B(n,r)$ corresponding to the vertex $i$. Hence we use this notation both for objects in $\mathcal{C}_{[0,n-1]}$ and $\modi B(n,r)$. The action of $x_j$ on  elements in $M_i$ corresponds to the action of $\gamma_j^{(i)}$ on elements in $e_iM$.\\
In the following, we thus regard $\modi B(n,r)$ as a full subcategory of $\modi_{\mathbb{Z}} kE_r$ and we will see in the next section that we gain a lot by viewing $\mathcal{C}_{[0,n-1]}$ as the module category for a bound quiver algebra. A general introduction to representation theory of quivers can be found in \cite{ass06}.\\
For all $0 \leq i \leq n-1$, we denote by $S(i)$, $P(i)$ and $I(i)$ the simple, the projective and the injective indecomposable $B(n,r)$-module corresponding to the vertex $i$.\\
Restricting $\mathfrak{F}$ to $\modi B(n,r)$ yields a functor
$$\mathfrak{F}_{(n,r)}: \modi B(n,r) \rightarrow \modi_n kE_r.$$

We now define subcategories of $\modi B(n,r)$ that correspond to the full subcategories \linebreak $\CR^j_n(kE_r) \subset \CR^j(kE_r) $, $\CJT_n(kE_r) \subset \CJT(kE_r)$, $\EIP_n(kE_r) \subset \EIP(kE_r)$ as well as \linebreak $\EKP_n(kE_r) \subset \EKP(kE_r)$ containing modules of Loewy length at most $n$.\\
Let therefore $\alpha \in k^r \backslash 0$. Now for the element $\tilde{\alpha}=\sum_{i=0}^{n-2}(\alpha_1 \gamma_1^{(i)} + \cdots + \alpha_r \gamma_r^{(i)}) \in B(n,r)$ and $M=\bigoplus_{i=0}^{n-1}M_i \in \modi B(n,r)$, left-multiplication with $\tilde{\alpha}$ yields a linear operator
$$\alpha_M : M \rightarrow M$$
such that for $1 \leq j \leq n-1$, $(\alpha_M)^j$ coincides with the left-multiplication with the element $\sum_{i=0}^{n-j-1}( (\alpha_1 \gamma_1^{(i+j-1)} + \cdots + \alpha_r \gamma_r^{(i+j-1)})  \cdots (\alpha_1 \gamma_1^{(i)} + \cdots + \alpha_r \gamma_r^{(i)}) ) \in B(n,r)$.

\begin{defi}For $n \leq p, r \geq 2$, we define full subcategories of $\modi B(n,r)$ as follows:
\begin{enumerate}[(a)]
\item $\EIP(n,r) := \left\{M \in \modi B(n,r)|  \im(\alpha_M)=\bigoplus_{i=1}^{n-1}M_i \ \ \forall \ \alpha \in k^r\backslash 0 \right\}$,
\item $\EKP(n,r) := \left\{M \in \modi B(n,r)|  \ker(\alpha_M)=M_{n-1} \ \ \forall \ \alpha \in k^r\backslash 0 \right\}$,
\item $\CR^j(n,r) := \left\{ M \in \modi B(n,r)| \exists c_j \in \mathbb{N}_0:  \rk (\alpha_{M})^j=c_j  \ \ \forall \ \alpha \in k^r\backslash 0 \right\}$,
\item $\CJT(n,r):= \bigcap_{j=1}^n  \CR^j(n,r)$.
\end{enumerate}
\end{defi}

\begin{rem}
\textnormal{A module $M \in \EIP(n,r)$ is standardly graded and $M_0=0$ implies $M=0$. Dually, a module $M \in \EKP(n,r)$ is costandardly graded and $M_{n-1}=0$ implies $M=0$.}
\end{rem}
\vspace{0,3cm}
Note that the duality $D: \modi B(n,r) \rightarrow \modi B(n,r)$ induced by relabelling the vertices in the reversed order and taking the linear dual is such that $D \EIP(n,r) = \EKP(n,r)$. Moreover observe that we have $\EIP(n,r) \cup \EKP(n,r) \subset \CR^j(n,r)$ for all $j \geq 1$.

\begin{prop}\label{func}
The restriction of $\mathfrak{F}_{(n,r)}$ to \newline $\mathcal{X} \in \left\{\EIP(n,r), \EKP(n,r), \CR^j(n,r), \CJT(n,r) \right\}$ induces a faithful exact functor
$$\mathfrak{F}_{\mathcal{X}}: \mathcal{X} \rightarrow \modi_n kE_r$$ such that
\begin{enumerate}[(i)]
\item for $\mathcal{X}=\EIP(n,r)$, $\mathfrak{F}_{\mathcal{X}}$ reflects isomorphisms and the essential image consists of the standardly gradable objects in $\EIP_n(kE_r)$.
\item for $\mathcal{X}=\EKP(n,r)$, $\mathfrak{F}_{\mathcal{X}}$ reflects isomorphisms and the essential image consists of the costandardly gradable objects in $\EKP_n(kE_r)$.
\item for $\mathcal{X}=\CR^j(n,r)$, the essential image of $\mathfrak{F}_{\mathcal{X}}$ consists of the $\left [0,n-1 \right]$-gradable objects in $\CR^j_n(kE_r)$.
\end{enumerate}
\end{prop}
\begin{proof}
Observe that for $M \in \modi B(n,r)$, the linear operator $\alpha_M$ given by $\alpha \in k^r \backslash 0$ corresponds to the linear operator $\alpha(t)_{\mathfrak{F}_{(n,r)}(M)}$ on $\mathfrak{F}_{(n,r)}(M)$ given by the $p$-point $\alpha$ with $\alpha(t)=\alpha_1x_1+ \cdots + \alpha_rx_r$.\\
$(i)$: With the preceding observation and in view of the fact that fibres of indecomposables are shifts on an indecomposable object \cite[4.1]{gg82}, it is easy to see that for an indecomposable object $M \in \modi_\mathbb{Z} kE_r$ we have $\mathfrak{F}_{(n,r)}(M) \in \EIP_n(kE_r)$ if and only if $M=N[i]$ for some object $N \in \EIP(n,r) \subset  \modi_{\mathbb{Z}}kE_r$.\\
Given $N \in \EIP(n,r)\backslash 0$, we have $\supp(N)=[0,l]$ for some $0 \leq l \leq n-1$. Thus we have $N[i] \notin \EIP(n,r)$ unless $i=0$ since $\supp N[i] =[i,l+i]$. Since $\mathfrak{F}_{(n,r)}$ commutes with direct sums and $\modi_n kE_r$ is a Krull-Schmidt category, $\mathfrak{F}_{\EIP(n,r)}$ thus reflects isomorphisms. Moreover, $M \in \EIP(n,r)$ is generated by $M_0$ and is hence standardly graded. Thus the essential image of $\mathfrak{F}_{\EIP(n,r)}$ consists of the standardly gradable objects in $\EIP_n(kE_r)$.\\
$(ii)$: Dual to $(i)$.\\
$(iii)$: Is clear in view of our general observation above.
\end{proof}
\begin{rem}
\textnormal{A direct consequence of Proposition \ref{func} is $\EIP(n,r) \cap \EKP(n,r) = (0)$, since $S(0)$ is the only simple $B(n,r)$-module in $\EIP(n,r)$, $S(n-1)$ the only simple $B(n,r)$-module in $\EKP(n,r)$ and $\EIP(kE_r) \cap \EKP(kE_r) = \add k$ \cite[4.4.3]{f11}.}
\end{rem}

\subsection{Homological characterization}
In this section, we will give a new point of view on our subcategories of $\modi B(n,r)$ which enables us to apply general methods from Auslander-Reiten theory.
The approach we present is inspired by work of Happel and Unger \cite{haun} on representations of the generalized Kronecker $\mathcal{K}_r$. The authors construct a representation $X=(X_1,X_2)$ over $\mathcal{K}_r$ corresponding to a given arrow $\gamma$ of $\mathcal{K}_r$  such that the representations $Y=(Y_1,Y_2)$ in the right-perpendicular category  $X^{\perp}$ are exactly those for which the operator $\gamma_Y: Y_1 \rightarrow Y_2$ corresponding to $\gamma$ is bijective \cite[2.1]{haun}.\\

Note that $P(i) \cong kE_r[i]/J^{n-i} kE_r[i]$ in $\modi_{\mathbb{Z}} kE_r$. Let $\alpha \in k^r \backslash 0$. Since $n-i \leq n-1 <p$, the map 
$$\alpha(i): P(i+1) \rightarrow P(i),\ e_{i+1} \mapsto \alpha_1\gamma^{(i)}_1 + \cdots +\alpha_r\gamma^{(i)}_r,$$ 
i.e. the right multiplication with $\alpha_1\gamma^{(i)}_1 + \cdots +\alpha_r\gamma^{(i)}_r$,
defines an embedding of $B(n,r)$-modules.
Composition yields embeddings 
$$\alpha(i)_j: P(i+j) \rightarrow P(i),\ e_{i+j} \mapsto ( \alpha_1\gamma^{(i+j-1)}_1 + \cdots + \alpha_r\gamma^{(i+j-1)}_r ) \cdots ( \alpha_1\gamma^{(i)}_1 + \cdots + \alpha_r\gamma^{(i)}_r )$$
for all $0 \leq i \leq n-2,\ 1 \leq j \leq n-i-1$. We let $ X_{\alpha}^{i,j}:=\coker \alpha(i)_j= P(i) / \alpha(i)_j(P(i+j))$.\\ 
For $1 \leq j \leq n-1$, $\alpha \in k^r \backslash 0$, we define 
$$X_{\alpha}^j=\bigoplus_{i=0}^{n-j-1} X_{\alpha}^{i,j}.$$
By definition, we obtain for the projective dimensions of these modules $\pd( X_{\alpha}^{i,j}) = 1$ and hence $\pd(X_{\alpha}^j)=1$. In the following, whenever we write $\Hom$ or $\Ext$, we refer to the category $\modi B(n,r)$.

\begin{theorem}\label{main}
We have
\begin{enumerate}[(a)]
\item $\EIP(n,r)= \left\{ M \in  \modi B(n,r)| \Ext^1(X_{\alpha}^1,M)=0 \; \forall \alpha \in k^r \backslash 0 \right\}$,
\item $\EKP(n,r)= \left\{ M \in  \modi B(n,r)| \Hom(X_{\alpha}^1,M)=0 \; \forall \alpha \in k^r \backslash 0 \right\}$,
\item $\CR^j(n,r)= \left\{ M \in  \modi B(n,r)| \exists c_j \ \dim_k \Ext^1(X_{\alpha}^j,M)=c_j\; \forall \alpha \in k^r \backslash 0 \right\}$.
\end{enumerate}
\end{theorem}
\begin{proof}
Consider the projective resolution
$0 \rightarrow P(i+j) \stackrel{\alpha(i)_j}{\longrightarrow} P(i) \rightarrow X_{\alpha}^{i,j} \rightarrow 0$
and, for $M \in \modi B(n,r)$, the exact sequence
$$0 \rightarrow \Hom(X_{\alpha}^{i,j}, M) \rightarrow \Hom(P(i), M) \rightarrow \Hom(P(i+j), M) \rightarrow \Ext^1(X_{\alpha}^{i,j} ,M) \rightarrow 0.$$
There is a commutative diagram

$$\begin{CD}
\Hom(P(i),M) @> \Hom(\alpha(i)_j,M)>> \Hom(P(i+j),M)\\
@VV \cong V @VV \cong V\\
M_i @> (\alpha_{M})^j|_{M_i}>> M_{i+j}
\end{CD}$$

whence
$$(\alpha_{M})^j|_{M_i}: M_{i} \rightarrow M_{i+j}$$
is surjective, resp. injective, if and only if $\Ext^1(X_{\alpha}^{i,j},M)=0$, resp. $\Hom(X_{\alpha}^{i,j},M)=0$. This already yields $(a)$ and $(b)$ and since $$\rk (\alpha_M)^j =\sum_{i=0}^{n-j-1} (\dim_k M_{i+j}-\dim_k\Ext^1(X_{\alpha}^{i,j},M))=\sum_{i=0}^{n-j-1}{\dim_k M_{i+j}}-\dim_k \Ext^1(X_{\alpha}^j,M),$$
we obtain $(c)$.
\end{proof}

Hence we have a homological description of the subcategories defined in $\S \ 2.1$ that involves a $\mathbb{P}^{r-1}$-family of $B(n,r)$-modules of projective dimension 1. At this juncture, we exploit fundamental homological properties of $\modi B(n,r)$ that do not hold in $\modi kE_r$.\\

Let us list some of the distinctive features of these modules.
\begin{prop}\label{prop} Let $\alpha \in k^r \backslash 0$ and let $\iota: B(n,r-1) \rightarrow B(n,r)$ be the embedding defined via $\gamma^{k}_l \mapsto \gamma^{k}_l$ for all $0 \leq k \leq n-2$ and $1 \leq l \leq r-1$.
\begin{enumerate}[(i)]
\item We have $\pd(X_{\alpha}^j)=1$ for all $1\leq j \leq n-1$.
\item  The module $X_{\alpha}^{i,j}$ is standardly graded and $\supp(X_{\alpha}^{i,j})=[i,n-1]$.
\item We have $\dim_k (X_{\alpha}^{i,j})_i=1$ and the module $X_{\alpha}^{i,j}$ is a brick in $\modi B(n,r)$.
\item All proper submodules of $X_{\alpha}^{n-2,1}$ are of the form $P(n-1)^{\oplus m}$ for some $m <r$.
\item The pullback  $\iota^*(X_{(0,\ldots,0,1)}^{i,1})$ is isomorphic to the projective $B(n,r-1)$-module $\tilde{P}(i)$.
\end{enumerate}
\end{prop}

In the following, we make use of Auslander-Reiten theory as well as torsion theory. At this point, we will briefly and in a somewhat informal way recall what Auslander-Reiten theory is about. A thorough introduction can be found in \cite[IV]{ass06}. The module category $\modi A$ is described in terms of the {\it Auslander-Reiten quiver} $\Gamma(A)$ of the algebra $A$ which is defined as follows:
\begin{enumerate}[(i)]
\item The vertices of $\Gamma(A)$ correspond to the isomorphism classes $[M]$ of indecomposable $A$-modules.
\item The arrows from $[N]$ to $[M]$ correspond to so-called {\it irreducible} maps $f: N \rightarrow M$, i.e. $f$ is neither a section nor a retraction and whenever $f=f_1f_2$, then either $f_1$ is a retraction or $f_2$ is a section.
\end{enumerate}
Each non-projective indecomposable module $M$ (non-injective indecomposable module $N$) gives rise to a uniquely determined short exact sequence, an {\it Auslander-Reiten} (or {\it almost split}) {\it sequence},
$$0 \rightarrow N \stackrel{f}{\rightarrow} \bigoplus_{i=1}^t E_i^{n_i} \stackrel{g}{\rightarrow} M \rightarrow 0$$
where $N$ ($M$) is indecomposable, the $E_i$ are pairwise non-isomorphic and indecomposable and the maps 
$f_{i_1},\ldots,f_{i_{n_i}}: N \rightarrow E_i$, $g_{i_1},\ldots, g_{i_{n_i}}: E_i \rightarrow M$ correspond to bases of the vector spaces of irreducible maps $N \rightarrow E_i$ and $E_i \rightarrow M$, respectively. We write $N=\tau(M)$ ($M=\tau^{-1}(N)$), where $\tau$ is referred to as the {\it Auslander-Reiten translation} of $M$ and we denote this in $\Gamma(A)$ by $[N] \dashleftarrow [M]$.\\
An indecomposable module $M$ is called {\it pre-projective} ({\it pre-injective}) if there is $n \in \mathbb{N}_0$ such that $\tau^n(M)$ ($\tau^{-n}(M)$) is projective (injective).
A module is referred to as {\it regular} if it is neither pre-injective nor pre-projective. Connected components of $\Gamma(A)$ that consist entirely of regular modules are then called regular.\\
Furthermore, we say that an indecomposable module $N$ is a {\it predecessor} ({\it successor}) of $M$ if there is a directed path from $[N]$ to $[M]$ ($[M]$ to $[N]$) in $\Gamma(A)$, i.e. a chain of irreducible maps from $N$ to $M$ ($M$ to $N$). We denote the set of all predecessors and successors by $(\rightarrow M)$ and $(M \rightarrow )$, respectively.\\

Given an algebra $A$, a pair $(\mathcal{T},\mathcal{F})$ of full subcategories of $\modi A$ is called {\it torsion pair} if the following conditions are satisfied
\begin{enumerate}[(a)]
\item $\Hom_A(M,N)=0$ for all $M \in \mathcal{T}$, $N \in \mathcal{F}$.
\item $\Hom_A(M,-)|_{\mathcal{F}}=0$ implies $M \in \mathcal{F}$.
\item $\Hom_A(-,N)|_{\mathcal{T}}=0$ implies $N \in \mathcal{F}$.
\end{enumerate}
The category $\mathcal{T}\ (\mathcal{F})$ is then called torsion (torsion-free) class of the torsion pair $(\mathcal{T},\mathcal{F})$.
According to \cite[VI, 1.4]{ass06}, torsion classes correspond to those full subcategories of $\modi A$ that are closed under images and extensions whereas torsion-free classes correspond to the full subcategories of $\modi A$ that are closed under submodules and extensions.

\begin{cor}\label{preinj}
The category $\EIP(n,r)$ is the torsion class $\mathcal{T}$ of a torsion pair $(\mathcal{T}, \mathcal{F})$ in $\modi B(n,r)$ with $\EKP(n,r) \subset \mathcal{F}$ that is closed under the Auslander-Reiten translate $\tau$ and which contains all preinjective modules.\\
Dually, $\EKP(n,r)$ is a torsion-free class $\mathcal{F}'$ of a torsion pair $(\mathcal{T'}, \mathcal{F'})$ in $\modi B(n,r)$  with $\EIP(n,r) \subset \mathcal{T}'$ that is closed under $\tau^{-1}$ and  contains all preprojective modules.\\
In particular, there are no non-trivial maps $\EIP(n,r) \rightarrow \EKP(n,r)$.
\end{cor}

\begin{proof}
Application of Theorem \ref{main} directly yields that $\EIP(n,r)$ is extension closed. Since $\pd(X_{\alpha}^j)=1$ and hence $\Ext^2(X_{\alpha}^j,-)=0$, the class is furthermore image closed. Thus $\EIP(n,r)$ is a torsion class in $\modi B(n,r)$.\\
The corresponding torsion-free objects in $\mathcal{F}=\left\{M \in \modi B(n,r)|\Hom(\mathcal{T},M)=0 \right\}$ are those that do not have any non trivial submodules in $\EIP(n,r)$. In particular, all $N \in \modi B(n,r)$ such that $N_0=0$ are torsion-free.\\
We now show that for $M \in  \EIP(n,r)$, we have $\tau(M) \in  \EIP(n,r)$. The Auslander-Reiten formula \cite[IV, 2.13]{ass06} yields an isomorphism
$$\Ext^1(X_{\alpha}^1,\tau M) \cong D\underline{\Hom}(M,X_{\alpha}^1),$$
where $$\Hom(M,X_{\alpha}^1) \cong \bigoplus_{i=0}^{n-1} \Hom(M, X_{\alpha}^{i,1}).$$
For $i \geq 1$, we have $(X_{\alpha}^{i,1})_0=0$ (Prop. \ref{prop}, (ii)) and thus $X_{\alpha}^{i,1} \in \mathcal{F}$. This yields the isomorphism
$\Hom(M,X_{\alpha}^1) \cong  \Hom(M, X_{\alpha}^{0,1} ).$
Since $[0] \subset [0,n-1]=\supp X_{\alpha}^{0,1}$ (Prop. \ref{prop}, (ii)) and by definition $\im \alpha_{X_{\alpha}^{0,1}}=0$, we in particular obtain $X_{\alpha}^{0,1} \notin \EIP(n,r)$.\\
Since $X_{\alpha}^{0,1}=B(n,r) (X_{\alpha}^{0,1})_0$  as well as $\dim_k (X_{\alpha}^{0,1})_0=1$, this already yields $X_{\alpha}^{0,1} \in \mathcal{F}$ and hence $\Hom(M, X_{\alpha}^{0,1})=0$ which implies $\tau(M) \in \EIP(n,r)$.\\
Moreover, Theorem \ref{main} directly yields that $\EIP(n,r)$ contains all injective objects in \linebreak $\modi B(n,r)$ and hence also their $\tau^m$-shifts for all $m \geq0$, i.e. all preinjectives. The dual statement follows using $D$. Hence $\EKP(n,r)$ is closed under taking submodules and $\EIP(n,r) \cap \EKP(n,r) =(0)$ implies $\EKP(n,r) \subset \mathcal{F}$.
\end{proof}
Note furthermore that the inclusions $\EKP(n,r) \subset \mathcal{F}$ and $\EIP(n,r) \subset \mathcal{T}'$ are proper. We have $X^{n-2,1}_{\alpha} \in \mathcal{F} \backslash \EKP(n,r)$ for example.

Corollary \ref{preinj} implies that a mesh in the Auslander-Reiten quiver of $\modi B(n,r)$

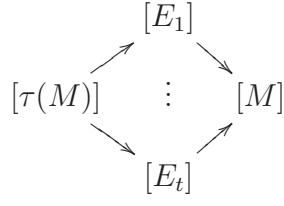
\[\begin{xy} 
\xymatrix@R=0.7em@C=0.7em{
& [E_1] \ar[rd] \\
[\tau(M)] \ar[ru] \ar[rd]& \vdots & [M]\\
& [E_t] \ar[ru]
}
\end{xy}\]

\vspace{0,5cm}
with $M$ in $\EIP(n,r)$, is completely contained in $\EIP(n,r)$. We thus obtain
\begin{cor}\label{pred}
Let $M \in \EIP(n,r)$ be indecomposable. Then  $(\rightarrow M) \subseteq \EIP(n,r)$. Dually, for $M \in \EKP(n,r)$, we have $(M \rightarrow) \subseteq \EKP(n,r)$.
\end{cor} 

We can make more precise statements for $\mathbb{Z}A_{\infty}$-components of $\Gamma(B(n,r))$. These components can be visualized as follows:
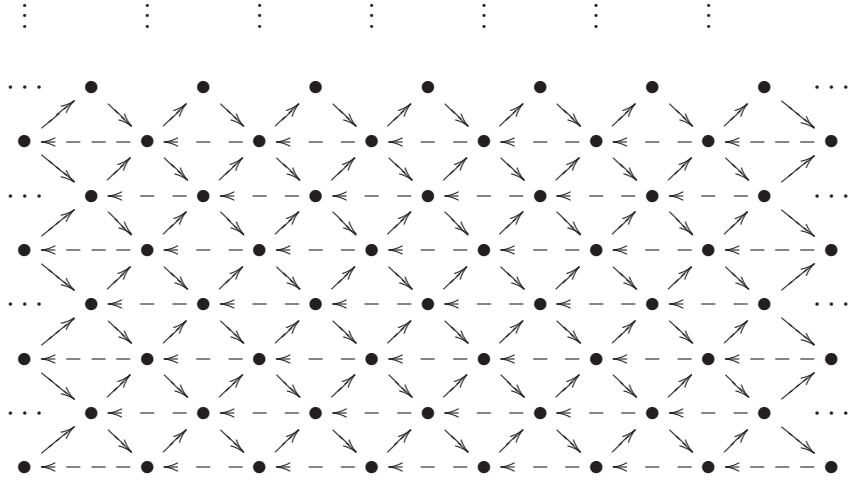
\[ \begin{xy}
\xymatrix@R=0.7em@C=0.7em{
\vdots  && \vdots && \vdots &&\vdots  && \vdots && \vdots &&\vdots\\
\cdots &   \bullet   \ar[rd]  &&  \bullet   \ar[rd]   && \bullet  \ar[rd]   &&  \bullet  \ar[rd] &&  \bullet  \ar[rd]  &&  \bullet  \ar[rd] && \bullet  \ar[rd]  & \cdots \\
 \bullet \ar[rd] \ar[ru] && \bullet \ar[rd] \ar[ru] \ar@{-->}[ll] && \bullet \ar[rd] \ar[ru] \ar@{-->}[ll] && \bullet \ar[rd] \ar[ru] \ar@{-->}[ll] && \bullet \ar[rd] \ar[ru] \ar@{-->}[ll] && \bullet \ar[rd] \ar[ru] \ar@{-->}[ll] && \bullet \ar[rd] \ar[ru]  \ar@{-->}[ll] &&\bullet  \ar@{-->}[ll] \\
  \cdots &  \bullet \ar[ru]  \ar[rd]  &&  \bullet \ar[ru]  \ar[rd] \ar@{-->}[ll]  &&  \bullet \ar[ru]  \ar[rd] \ar@{-->}[ll]  &&  \bullet \ar[ru]  \ar[rd] \ar@{-->}[ll] &&  \bullet \ar[ru]  \ar[rd] \ar@{-->}[ll]  &&  \bullet \ar[ru]  \ar[rd] \ar@{-->}[ll]   && \bullet   \ar[rd] \ar@{-->}[ll] \ar[ru] & \cdots \\
\bullet \ar[ru]  \ar[rd] &&   \bullet \ar[ru]  \ar[rd] \ar@{-->}[ll] &&  \bullet  \ar[ru]  \ar[rd] \ar@{-->}[ll]  &&  \bullet  \ar[ru]  \ar[rd] \ar@{-->}[ll] &&  \bullet  \ar[ru]  \ar[rd] \ar@{-->}[ll] &&  \bullet  \ar[ru]  \ar[rd] \ar@{-->}[ll]&&  \bullet  \ar[rd] \ar@{-->}[ll] \ar[ru] && \bullet  \ar@{-->}[ll] \\
{ \cdots} & \bullet\ar[ru]  \ar[rd]  &&   \bullet \ar[ru]  \ar[rd] \ar@{-->}[ll]  &&  \bullet \ar[ru]  \ar[rd] \ar@{-->}[ll]  &&  \bullet \ar[ru]  \ar[rd] \ar@{-->}[ll] &&  \bullet \ar[ru]  \ar[rd] \ar@{-->}[ll]   && \bullet   \ar[ru]  \ar[rd] \ar@{-->}[ll] && \bullet  \ar[ru]  \ar[rd] \ar@{-->}[ll] & { \cdots} \\
\bullet \ar[rd]  \ar[ru] && \bullet \ar[rd]  \ar[ru] \ar@{-->}[ll] &&  \bullet \ar[rd]  \ar[ru] \ar@{-->}[ll] && \bullet \ar[rd]  \ar[ru] \ar@{-->}[ll] && \bullet \ar[rd]  \ar[ru] \ar@{-->}[ll] &&  \bullet \ar[rd]  \ar[ru] \ar@{-->}[ll] && \bullet \ar[ru] \ar[rd] \ar@{-->}[ll] && \bullet  \ar@{-->}[ll] \\
{ \cdots} & \bullet \ar[ru]  \ar[rd]  &&  \bullet \ar[ru]  \ar[rd] \ar@{-->}[ll]  &&   \bullet \ar[ru]  \ar[rd] \ar@{-->}[ll]  &&  \bullet \ar[ru]  \ar[rd] \ar@{-->}[ll] &&  \bullet \ar[ru]  \ar[rd] \ar@{-->}[ll]   && \bullet  \ar[ru]  \ar[rd] \ar@{-->}[ll] && \bullet \ar[rd] \ar[ru]  \ar@{-->}[ll] & { \cdots} \\
\bullet  \ar[ru] &&  \bullet  \ar[ru] \ar@{-->}[ll] &&  \bullet  \ar[ru] \ar@{-->}[ll] &&   \bullet \ar[ru] \ar@{-->}[ll] &&  \bullet \ar[ru] \ar@{-->}[ll]&&  \bullet \ar[ru] \ar@{-->}[ll] && \bullet \ar[ru] \ar@{-->}[ll]&& \bullet \ar@{-->}[ll]}
\end{xy} \]

Modules in the bottom row of such components are called {\it quasi-simple}. Ringel \cite{ri78} has shown that for each module $M$ in a regular $\mathbb{Z}A_{\infty}$-component $\mathcal{C}$, there exist uniquely determined quasi-simple modules $X$ and $Y$ $\in \mathcal{C}$ and uniquely determined chains of irreducible monomorphisms $X=X_1  \rightarrow \cdots  \rightarrow X_{s-1}\rightarrow X_s= M$ and epimorphisms $M=Y_s \rightarrow Y_{s-1} \rightarrow  \cdots \rightarrow Y_1=Y$ where $s$ is the so called {\it quasi-length} of $M$ and $X$ ($Y$) is referred to as the {\it quasi-socle} ({\it quasi-top}) of $M$. Moreover, $M$ is uniquely determined by its quasi-length and quasi-socle (quasi-top) whence we write $M=X(s)$ ($M=[s]Y$).\\
\begin{prop}\label{ZA}
Let $\mathcal{C}$ be a regular $\mathbb{Z}A_{\infty}$-component of $\Gamma(B(n,r))$. If $\EIP(n,r) \cap \mathcal{C} \neq \emptyset$, then either $\mathcal{C} \subseteq \EIP(n,r)$ or there exists a quasi-simple module $W_{\mathcal{C}}$ such that \linebreak
 $(\rightarrow W_{\mathcal{C}})=\mathcal{C} \cap \EIP(n,r).$
Dually, if $\EKP(n,r) \cap \mathcal{C} \neq \emptyset$, then either $\mathcal{C} \subseteq \EKP(n,r)$ or there exists a quasi-simple module $M_{\mathcal{C}}$ such that
$(M_{\mathcal{C}} \rightarrow)=\mathcal{C} \cap \EKP(n,r).$
\end{prop}
\begin{proof}
Since in every regular $\mathbb{Z}A_{\infty}$-component the irreducible maps from top to bottom are surjective, $\EIP(n,r) \cap \mathcal{C} \neq \emptyset$ yields the existence of a quasi-simple module $W$ in $\mathcal{C}$ that belongs to $\EIP(n,r)$. If all quasi-simple modules belong to $\EIP(n,r)$, Corollary \ref{pred} yields $\mathcal{C} \subset \EIP(n,r)$. In view of Corollary \ref{pred} and the fact that any two quasi-simple modules are successor, resp. predecessor of one another, we can choose $k$ maximal such that $W_{\mathcal{C}}:=\tau^{-k}(W) \in \EIP(n,r)$ and $(\rightarrow W_{\mathcal{C}})=\mathcal{C} \cap \EIP(n,r)$. Dual properties of modules in $\EKP(n,r)$ yield the assertion.
\end{proof}
Furthermore, we can extend Corollary \ref{preinj} with regard to when the translate of a module satisfies the equal images property.

\begin{prop}
Let $M=M_0 \oplus M_1 \cdots \oplus M_{n-1} \in \modi B(n,r)$ be indecomposable and generated by $M_0$. If $M_{n-1}=0$, then $\tau M \in \EIP(n,r)$.
\end{prop}
\begin{proof}
Applying Theorem \ref{main} again in combination with the Auslander-Reiten formula, it suffices to show that $\Hom(M,X_{\alpha}^1)=0$. Since the module $M$ is generated by $M_0$, we have $\Hom(M,X_{\alpha}^1) \cong \Hom(M,X_{\alpha}^{0,1})$.\\
Assume that there is a non-trivial morphism $\varphi: M \rightarrow X_{\alpha}^{0,1}$. Then there exists $m \in M_0$ such that $\varphi(m) \in (X_{\alpha}^{0,1})_0 \backslash 0$. Since $(X_{\alpha}^{0,1})_0$ is one-dimensional (Proposition \ref{prop}, (iii)) and $X_{\alpha}^{0,1}$ is generated by $(X_{\alpha}^{0,1})_0$ (Proposition \ref{prop}, (ii)), $\varphi$ is hence surjective. This contradicts the fact that $M_{n-1}=0$ and $(X_{\alpha}^{0,1})_{n-1} \neq 0$.
\end{proof}

The next result concerns the special role that $W$- and $M$- modules play as modules for generalized Beilinson algebras.\\
Recall that for all $m \in \mathbb{N}, d \leq n$, the $\mathbb{Z}$-graded module $M_{m,d}^{(r)}$ endowed with the grading from $\S \ 1$ satisfies $\supp(M_{m,d}^{(r)})=[m-d,m-1]$. Hence we have $M_{m,d}^{(r)}[n-m] \in \mathcal{C}_{[0,n-1]}$ such that $M_{m,d}^{(r)}[n-m]$ is an object in $\EKP(n,r)$. Likewise, the canonical $\mathbb{Z}$-grading on $W$-modules is such that $\supp(W_{m,d}^{(r)})=[-m+1,-m+d]$ and hence $W_{m,d}^{(r)}[m-1] \in \mathcal{C}_{[0,n-1]}$ is an object in $\EIP(n,r)$. For our duality $D$ on $\modi B(n,r)$, we have 
$$DM_{m,d}^{(r)}[n-m] \cong W_{m,d}^{(r)}[m-1].$$ 
Note furthermore that for $1 \leq d \leq n$, we have $M_{d,d}^{(r)}[n-d] \cong P(n-d)$ and $W_{d,d}^{(r)}[d-1] \cong I(d-1)$. Since $M_{m,d}^{(r)}$ is a brick in $\modi_{\mathbb{Z}} kE_r$ by Corollary \ref{brick}, $M_{m,d}^{(r)}[n-m]$ is a brick in $\modi B(n,r)$.\\
In the remainder of this section, we are concerned with $B(n,r)$-modules and hence shorten notation and write $M_{m,d}^{(r)}$ for the $B(n,r)$-module $M_{m,d}^{(r)}[n-m]$ and likewise $W_{m,d}^{(r)}$ for the $B(n,r)$-module $W_{m,d}^{(r)}[m-1]$.\\
The following theorem does not hold in case $r=2$. Since modules of the form $M_{m,2}^{(2)}$ are preprojective, we have $\tau(M_{m,2}^{(2)}) \in \EKP(2,2) \backslash 0$ for $m>2$.

\begin{theorem}\label{wmod}
Let $r \geq 3$ and let $n \leq p$, $m > n$. Then $\tau (M_{m,n}^{(r)}) \in \EIP(n,r)$ and dually $\tau^{-1}(W_{m,n}^{(r)}) \in \EKP(n,r)$.
\end{theorem}

\begin{proof}
We want to apply Theorem \ref{main} again in combination with the Auslander-Reiten formula and thus show that for all $\alpha \in k^r \backslash 0$, there are only trivial maps $M_{m,n}^{(r)} \rightarrow X_{\alpha}^1$. Since $M_{m,n}^{(r)}$ is generated by $(M_{m,n}^{(r)})_0$, we have $\Hom(M_{m,n}^{(r)},X_{\alpha}^1) \cong \Hom(M_{m,n}^{(r)},X_{\alpha}^{0,1})$ and a non-trivial map $\varphi: M_{m,n}^{(r)} \rightarrow X_{\alpha}^{0,1}$ is necessarily surjective (\ref{prop}, (ii), (iii)). By Proposition \ref{stable}, $M_{m,n}^{(r)}$ is $\GL_r(k)$-stable. Furthermore, for all $\alpha$, $\beta \in k^r \backslash 0$, there exists $g \in \GL_r(k)$ such that $(X_{\alpha}^{0,1})^{(g)} \cong X_{\beta}^{0,1}$. Since $\Hom(M_{m,n}^{(r)}, X_{\alpha}^{0,1}) \cong \Hom(M_{m,n}^{(r)}, (X_{\alpha}^{0,1})^{(g)})$ we may hence assume that $\alpha=(0,0,\ldots,1)$.\\
Now we have $\gamma^{(i)}_r \in \ann_{B(n,r)}X_{\alpha}^{0,1}$ and thus $\gamma^{(i)}_r  M_{m,n}^{(r)} \subseteq \ker \varphi$ for all $0 \leq i \leq n-2$. Note that $N:=\sum_{i=0}^{n-2}\gamma^{(i)}_r  M_{m,n}^{(r)}$ is a submodule of $M$ such that $\gamma_r^{(i)}$ acts trivially on $\tilde{M}:=M_{m,n}^{(r)}/N$. Observe that for the embedding $\iota$ from Proposition \ref{prop}, we have $\iota^*(\tilde{M}) \cong M_{m,n}^{(r-1)}$. Moreover, Proposition \ref{prop}, (v), yields $\iota^*(X_{\alpha}^{0,1}) \cong  \tilde{P}(0)$ for the projective indecomposable of $B(n,r-1)$-module corresponding to the vertex $0$. \\
Thus there results a split epimorphism
$M_{m,n}^{(r-1)} \rightarrow \tilde{P}(0)$  of $B(n,r-1)$-modules
which is a contradiction since by Theorem \ref{brick} $M_{m,n}^{(r-1)}$ is indecomposable and furthermore non-projective in $\modi B(n,r-1)$ since $m>n$ and $r>2$. Hence we have $\tau (M_{m,n}^{(r)}) \in \EIP(n,r)$. Our duality $D$ on $\modi B(n,r)$ now yields the assertion.
\end{proof}

\begin{lemma}\label{cjt}
Let $0 \rightarrow A \rightarrow B \rightarrow C \rightarrow 0$ be an exact sequence in $\modi B(n,r)$. If $A \in \EIP(n,r)$, then $B \in \CR^j(n,r)$ if and only if $C \in \CR^j(n,r)$.
\end{lemma}
\begin{proof}
Since $\Ext^2(X_{\alpha}^j,-)=0$, we get an exact sequence
$$ \Ext^1(X_{\alpha}^j,A) \rightarrow \Ext^1(X_{\alpha}^j,B) \rightarrow \Ext^1(X_{\alpha}^j,C)\rightarrow 0,$$
where $\Ext^1(X_{\alpha}^j,A)=0$ since $A \in \EIP(n,r)$. Thus the dimension of the rightmost term does not depend on $\alpha$ iff the dimension of the middle term does not.
\end{proof}
 
We close this section on Beilinson algebras with the following statement concerning Auslander-Reiten sequences. In case $n=2$, this is a direct consequence of Theorem $\ref{shape}$ below.
\begin{prop}\label{indec}
Let $0 \rightarrow A \rightarrow B \rightarrow C \rightarrow 0$ be an Auslander-Reiten sequence in $\modi B(n,r)$ such that $A$ is in $\EIP(n,r)$  and $C$ is in $\EKP(n,r)$. Then $B$ is an indecomposable module in $\CJT(n,r) \backslash ( \EKP(n,r) \cup \EIP(n,r))$.
\end{prop}

\begin{proof}
Let us first of all show that $B$ is indecomposable. Assume that there exists a decomposition $B=\oplus_{i \in I} B_i$ such that $|I|>2$. Then for reasons of dimension it is not possible that all irreducible maps $A \rightarrow B_i$ are injective and all irreducible maps $B_j \rightarrow C$ are surjective (this would imply $\dim A + \dim C < \dim B_i + \dim B_j \leq \dim B$). Thus there exists an epimorphism $A \rightarrow B_i$ for some $i$ or a monomorphism $B_i \rightarrow C$. This now implies that $B_i$ satisfies the equal images property, respectively the equal kernels property. Now in case $B_i \in \EIP(n,r)$, every morphism $B_i \rightarrow C$ is trivial in view of Corollary \ref{preinj}. With the same argument $B_i \in \EKP(n,r)$ yields that every morphism $A \rightarrow B_i$ is trivial, a contradiction. Thus $B$ is indecomposable. Corollary \ref{preinj} yields $B \notin \EIP(n,r) \cup \EKP(n,r)$, whereas $B \in \CJT(n,r)$ follows from Lemma \ref{cjt}.
\end{proof}

\begin{rem}
\textnormal{Returning to the categories $\EIP(kE_r),\ \EKP(kE_r),\ \CR^j(kE_r)$ and $\CJT(kE_r)$, Lemma \ref{cjt} holds in $\modi kE_r$ as well and follows directly from the Snake Lemma \cite[I.5, 5.1]{ass06}. Demanding that $k$ is not a direct summand of the middle-term $B$, Proposition \ref{indec} also holds in $\modi kE_r$.}
\end{rem}

\section{The generalized Kronecker quiver}
We are now going to confine our investigations to the case $B(2,r)$ where $r \geq 2$. The algebra $B(2,r)$ is isomorphic to $\mathcal{K}_r$, the path algebra of the $r$-Kronecker quiver. Note that $\mathfrak{F}_{(2,r)}$ is dense and so are the functors $\mathfrak{F}_{\mathcal{X}}$ from Proposition \ref{func}. Furthermore, we have $\CJT(2,r)=\CR^1(2,r)$. As was mentioned above, the indecomposable equal images modules for $kE_2$ of Loewy length at most 2 have been classified in \cite{cfs09}: The only indecomposable modules in $\EIP(2,2)$ are the preinjective modules over $\mathcal{K}_2$, i.e. the modules $W_{n,2}$ and the simple injective module $S(1)$ \cite[4.2.2]{f11}. This implies that $\EIP(2,2)$ is the additive closure of the preinjectives. Apart from the preprojective modules, that satisfiy the equal kernels property, there are no other indecomposable modules of constant Jordan type. We show that the situation is completely different for $r \geq 3$.\\

The algebra $\mathcal{K}_r$ is wild if $r >2$ and tame if $r=2$.
Recall that the Auslander-Reiten translate for the hereditary algebra $\mathcal{K}_r$, $r \geq 2$, is given by $\tau \cong \Ext^1 (-,\mathcal{K}_r)^*$ \cite[V.II, 1.9]{ass06}.  The components of the Auslander-Reiten quiver $\Gamma( \mathcal{K}_r)$ of $\modi \mathcal{K}_r$ have the following shape if $r >2$:\\
There is exactly one preprojective component $\mathcal{P}$, consisting of the two projective modules and their $\tau^{-1}$-shifts and exactly one preinjective component $\mathcal{I}$ consisting of the two injective modules and their $\tau$-shifts. Ringel has proven in \cite{ri78}, that the remaining (regular) components are of type $\mathbb{Z}A_{\infty}$.\\

From now on, we assume that $r > 2$. We write $X_{\alpha}:=X_{\alpha}^1=\coker \alpha(1)$ for $\alpha \in k^r \backslash 0$. The module $X_{\alpha}$ is a brick and has no proper submodules apart from direct sums of $P(1)$ (Proposition \ref{prop}, (iii), (iv)). Via computing the dimension vectors of the preprojective and preinjective modules, we can conclude that $X_{\alpha}$ is regular and, since it has no proper regular submodules, thus quasi-simple. Moreover, computation yields 
\begin{prop}\label{dual} Let $\alpha \in k^r \backslash 0$. Then
$D X_{\alpha}=\Ext^1(X_{\alpha}, \mathcal{K}_r)^*=\tau X_{\alpha}$.
\end{prop}
According to Thereom \ref{preinj}, we have $\mathcal{I} \subseteq \EIP(2,r)$ and $\mathcal{P} \subseteq \EKP(2,r)$ whereas Theorem \ref{wmod} implies that $W_{n,2}^{(r)} \notin \mathcal{I}$ and $M_{n,2}^{(r)} \notin \mathcal{P}$ for $n>2$. Thus these modules are examples of regular modules with the equal images property and with the equal kernels property, respectively. 

\subsection{Regular components}
We will now describe the occurrence of regular equal images and equal kernels modules in the Auslander-Reiten quiver $\Gamma( \mathcal{K}_r)$ of $\modi \mathcal{K}_r$. 

In order to show the existence of equal images as well as equal kernels modules in every regular component of $\Gamma( \mathcal{K}_r)$, we record the following dual version of a lemma by Kerner:
\begin{lemma}[Kerner \cite{ker}, 4.6]\label{kerner}
If $X, Y$ are regular modules over a wild hereditary algebra, there exists an integer $N$ with $\Hom(Z,\tau^{-m}(X))=0$ for all $m \geq N$ and all regulars $Z$ with $\dim_k Z \leq \dim_k Y$.
\end{lemma}
Note that our next result also follows from Corollary \ref{preinj} in combination with \cite[Theorem (B)]{assker96}, a general result concerning non-splitting torsion pairs for wild hereditary algebras.
\begin{theorem}\label{shape}
Let $\mathcal{C}$ be a regular component of $\Gamma(B(2,r))$. Then $\mathcal{C}$ contains two uniquely determined quasi-simple modules $W_{\mathcal{C}}$ and $M_{\mathcal{C}}$ such that
$$(\rightarrow W_{\mathcal{C}})=\mathcal{C} \cap \EIP(2,r) \text{ and }
(M_{\mathcal{C}} \rightarrow)=\mathcal{C} \cap \EKP(2,r).$$
\end{theorem}
\begin{proof}
Let $\mathcal{C}$ be a regular component, $X$ be in $\mathcal{C}$. Since we have $\dim_k X_{\alpha} = \dim_k X_{\beta}$ for all $\alpha, \beta \in k^r \backslash 0$ and $\mathcal{K}_r$ is wild, we can apply Lemma \ref{kerner} with $Y=X_{\alpha}$ for some $\alpha$ and $Z$ running through all $X_{\beta}$, $\beta \in k^r \backslash 0$. This implies that there exists an $N$ such that $\Hom_{\mathcal{K}_r}(X_{\alpha},\tau^{-m}(X))=0$ for all $m \geq N$ and all $\alpha \in  k^r \backslash 0$. In view of Theorem \ref{main} we thus have $\tau^{-m}(X) \in \EKP(2,r)$ for all $m \geq N$. Dually, $\EIP(2,r) \cap \mathcal{C} \neq \emptyset$. Now apply Proposition \ref{ZA}.
\end{proof}

Thus the regular components of $\Gamma(\mathcal{K}_r)$ have the following shape

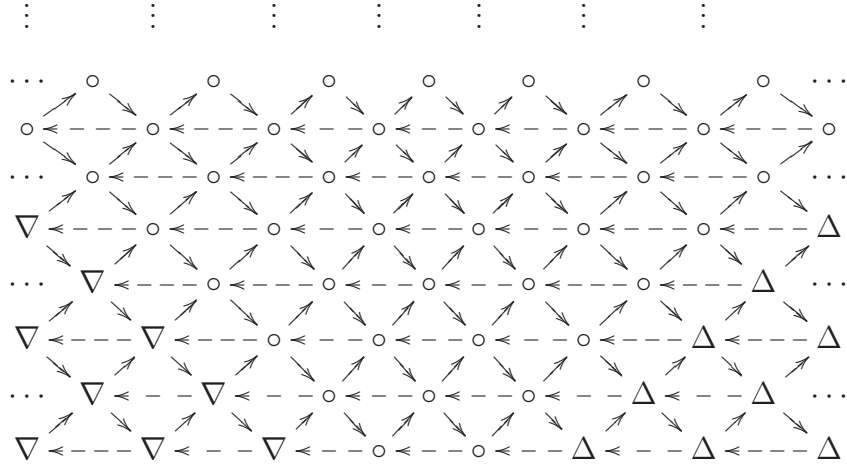
\[ \begin{xy}
\xymatrix@R=0.5em@C=0.5em{
\vdots  && \vdots && \vdots &&\vdots  && \vdots && \vdots &&\vdots\\
\cdots &   \circ   \ar[rd]  &&  \circ   \ar[rd]   && \circ  \ar[rd]   &&  \circ  \ar[rd] &&  \circ  \ar[rd]  &&  \circ  \ar[rd] && \circ  \ar[rd]  & \cdots \\
 \circ \ar[rd] \ar[ru] && \circ \ar[rd] \ar[ru] \ar@{-->}[ll] && \circ \ar[rd] \ar[ru] \ar@{-->}[ll] && \circ \ar[rd] \ar[ru] \ar@{-->}[ll] && \circ \ar[rd] \ar[ru] \ar@{-->}[ll] && \circ \ar[rd] \ar[ru] \ar@{-->}[ll] && \circ \ar[rd] \ar[ru]  \ar@{-->}[ll] &&\circ  \ar@{-->}[ll] \\
  \cdots &  \circ \ar[ru]  \ar[rd]  &&  \circ \ar[ru]  \ar[rd] \ar@{-->}[ll]  &&  \circ \ar[ru]  \ar[rd] \ar@{-->}[ll]  &&  \circ \ar[ru]  \ar[rd] \ar@{-->}[ll] &&  \circ \ar[ru]  \ar[rd] \ar@{-->}[ll]  &&  \circ \ar[ru]  \ar[rd] \ar@{-->}[ll]   && \circ   \ar[rd] \ar@{-->}[ll] \ar[ru] & \cdots \\
\nabla \ar[ru]  \ar[rd] &&   \circ \ar[ru]  \ar[rd] \ar@{-->}[ll] &&  \circ  \ar[ru]  \ar[rd] \ar@{-->}[ll]  &&  \circ  \ar[ru]  \ar[rd] \ar@{-->}[ll] &&  \circ  \ar[ru]  \ar[rd] \ar@{-->}[ll] &&  \circ  \ar[ru]  \ar[rd] \ar@{-->}[ll]&&  \circ  \ar[rd] \ar@{-->}[ll] \ar[ru] && \Delta  \ar@{-->}[ll] \\
 \cdots & \nabla \ar[ru]  \ar[rd]  &&   \circ \ar[ru]  \ar[rd] \ar@{-->}[ll]  &&  \circ \ar[ru]  \ar[rd] \ar@{-->}[ll]  &&  \circ \ar[ru]  \ar[rd] \ar@{-->}[ll] &&  \circ \ar[ru]  \ar[rd] \ar@{-->}[ll]   && \circ   \ar[ru]  \ar[rd] \ar@{-->}[ll] && \Delta  \ar[ru]  \ar[rd] \ar@{-->}[ll] &  \cdots \\
\nabla  \ar[rd]  \ar[ru] && \nabla  \ar[rd]  \ar[ru] \ar@{-->}[ll] &&  \circ \ar[rd]  \ar[ru] \ar@{-->}[ll] && \circ \ar[rd]  \ar[ru] \ar@{-->}[ll] && \circ \ar[rd]  \ar[ru] \ar@{-->}[ll] &&  \circ \ar[rd]  \ar[ru] \ar@{-->}[ll] && \Delta \ar[ru] \ar[rd] \ar@{-->}[ll] && \Delta \ar@{-->}[ll] \\
 \cdots & \nabla  \ar[ru]  \ar[rd]  &&  \nabla  \ar[ru]  \ar[rd] \ar@{-->}[ll]  &&   \circ \ar[ru]  \ar[rd] \ar@{-->}[ll]  &&  \circ \ar[ru]  \ar[rd] \ar@{-->}[ll] &&  \circ \ar[ru]  \ar[rd] \ar@{-->}[ll]   && \Delta  \ar[ru]  \ar[rd] \ar@{-->}[ll] && \Delta\ar[rd] \ar[ru]  \ar@{-->}[ll] & \cdots \\
\nabla   \ar[ru] && \nabla   \ar[ru] \ar@{-->}[ll] &&  \nabla   \ar[ru] \ar@{-->}[ll] &&   \circ \ar[ru] \ar@{-->}[ll] &&  \circ \ar[ru] \ar@{-->}[ll]&& \Delta \ar[ru] \ar@{-->}[ll] && \Delta \ar[ru] \ar@{-->}[ll]&& \Delta \ar@{-->}[ll]}
\end{xy} \]

where $\nabla$ and $\Delta$ indicate that the corresponding module is an object in $\EIP(2,r)$, resp. in $\EKP(2,r)$. Hence for each regular component $\mathcal{C}$, the {\it width} $\mathcal{W}(\mathcal{C})$ of the gap between these two modules, i.e. the natural number $k$ such that $\tau^{k+1}(M_{\mathcal{C}})=W_{\mathcal{C}}$ is an invariant for $\mathcal{C}$.\\

{\bf Examples}
\begin{enumerate}
\item Let $\mathcal{C}_n$ be the component containg the module $W_{n,2}^{(r)}$ for $n > 2$. By Theorem \ref{wmod} we have $\tau^{-1}(W_{n,2}^{(r)}) \in \EKP(2,r)$ and thus $W_{n,2}^{(r)}=W_{\mathcal{C}_n}$ and $\tau^{-1}(W_{n,2}^{(r)})=M_{\mathcal{C}_n}$. Hence $\mathcal{W}(\mathcal{C}_n)=0$.

\item Let $\mathcal{C}_{\alpha}$ be the component containing the quasi-simple brick $X_{\alpha}$. Recall that by Proposition \ref{dual}, we have $\tau (X_{\alpha}) \cong D X_{\alpha}$. Since $\mathcal{K}_r$ is wild hereditary, $\tau$ is an equivalence on the full subcategory of regular modules and hence
$$\Hom_{\mathcal{K}_r}(X_{\beta},\tau^{-1}(X_{\alpha})) \cong \Hom_{\mathcal{K}_r}(\tau(X_{\beta}),X_{\alpha}) \cong \Hom_{\mathcal{K}_r}(DX_{\beta},X_{\alpha})$$ 
for all $\beta \in k^r \backslash 0$. Proper submodules of $X_{\alpha}$ are of the form $P(2)^{\oplus m}$  (Proposition \ref{prop}, (iv)) and, dually, proper factor modules of $DX_{\beta}$ are of the form $I(1)^{\oplus m'}$. Hence the rightmost term is equal to zero. According to Theorem \ref{main}, we thus obtain $\tau^{-1}(X_{\alpha}) \in \EKP(2,r)$. Using the Auslander-Reiten formula, we can analogously show that $\tau^2(X_{\alpha}) \in \EIP(2,r)$. Hence $\mathcal{W}(C_{\alpha})=2.$

\item Let $\mathcal{C}_{\lambda}$ be the component containing the brick $E^{(\lambda)}$ for $\lambda \in k^r \backslash 0$ with dimension vector $(1,1)$  on which $\gamma_i$ acts via multiplication with $\lambda_i$. Using the Auslander-Reiten formula in combination with Theorem \ref{main} we have
$$\Ext^1_{\mathcal{K}_r}(X_{\alpha},\tau(E^{(\lambda)})) \cong \Hom_{\mathcal{K}_r} (E^{(\lambda)},X_{\alpha})=0$$
and hence $\tau(E^{(\lambda)}) \in \EIP(n,r)$. Dualizing yields
\begin{align*}
\Hom_{\mathcal{K}_r}(X_{\alpha},\tau^{-1}(E^{(\lambda)}))& \cong \Hom_{\mathcal{K}_r}(DX_{\alpha},E^{(\lambda)})\\& \cong \Hom_{K_r}(D(E^{(\lambda)}),X_{\alpha})\\
& \cong \Hom_{\mathcal{K}_r} (E^{(\frac{1}{\lambda})},X_{\alpha})=0
\end{align*}
where $(\frac{1}{\lambda})_i=\frac{1}{\lambda_i}$ if $\lambda_i \neq 0$ and $(\frac{1}{\lambda})_i=0$ else, for all $1 \leq i \leq r$.
Hence $\mathcal{W} (C_{\lambda})=1$.

\end{enumerate}

The examples show that $\mathcal{W}(\mathcal{C})$ indeed varies while running through the different regular components and we will show that there is no upper boundary for this number.
\vspace*{5mm}

In \cite{ker90}, Kerner has defined an invariant for regular components of a wild hereditary algebra $A$.
 Let $\mathcal{C}$ be a regular component of $A$ and $X$ some quasi-simple module in $\mathcal{C}$. The quasi-rank of $\mathcal{C}$ is defined via
$$\rk \mathcal{C} = \min \left\{m \in \mathbb{Z}| \rad (X,\tau^l X) \neq 0 \; \forall l \geq m \right\},$$
where for two indecomposable modules $X,Y \in \modi \mathcal{K}_r$, $\rad (X,Y)$ is the vector space of all non-isomorphisms from $X$ to $Y$ (cf. \cite[A.3, 3.5]{ass06}). Hence for $l \neq 0$ and $X$ regular, it is $\rad (X,\tau^l X)= \Hom (X,\tau^l X)$. A theorem by Hoshino (cf. \cite[V]{cb}) says that for $A=\mathcal{K}_r$, $\rk$ is bounded above by 1.  In view of \cite[1.6]{ker90}, we can conclude that $\mathcal{C}$ contains a brick if and only if $\rk \mathcal{C}=1$.

\begin{prop}\label{rkW}~
\begin{enumerate}[(i)]
\item Let $\mathcal{C}$ be a regular component of $K_r$. If $\mathcal{C}$ does not possess a brick, then we have $|\rk \mathcal{C}| \leq \mathcal{W} (\mathcal{C})$.
\item Let $n \in \mathbb{N}$. Then there exists a regular component $\mathcal{C}$ of $\mathcal{K}_r$ such that $\mathcal{W}(\mathcal{C}) > n$.
\end{enumerate}
\end{prop}
\begin{proof}
$(i)$: Choose the quasi-simple module $W_{\mathcal{C}}$ in $\mathcal{C}$ given by Theorem \ref{shape}. The module $\tau^{- \mathcal{W} (\mathcal{C}) -1}(W_{\mathcal{C}})=M_C$ satisfies the equal kernels property and hence by Corollary \ref{preinj}
$\Hom(W_{\mathcal{C}},\tau^{- \mathcal{W} (\mathcal{C}) -1}(W_{\mathcal{C}})) =0$,
which implies $\rk {\mathcal{C}}> - \mathcal{W} (\mathcal{C}) -1$. Since $\mathcal{C}$ does not possess a brick and hence $\rk \mathcal{C} \leq 0$ it is $|\rk \mathcal{C}| \leq \mathcal{W}(\mathcal{C})$.\\
$(ii)$: In \cite[3.1]{kerlu} it is  proven that
$$\inf \left\{\rk(\mathcal{C}) |\; \mathcal{C} \in \Omega(\mathcal{K}_r) \right\}=- \infty$$
where $\Omega(\mathcal{K}_r)$ denotes the set of regular components of $\modi \mathcal{K}_r$. Since $\rk \mathcal{C}=1$ iff $\mathcal{C}$ contains a brick, we can conclude $(ii)$ with $(i)$.
\end{proof}

For every component $\mathcal{C}$ containg a brick, it is $\rk \mathcal{C}=1$. By contrast, the examples $\mathcal{C}_n$ and $\mathcal{C}_{\alpha}$ show, that some components containing bricks may be distinguished via the invariant $\mathcal{W}$.

\subsection{The category CJT(2,r)}
In this subsection, we direct our attention towards the category $\CJT(2,r)$ and make some statemens concerning Auslander-Reiten components of $B(2,r)$. Friedlander and Pevtsova have shown that for the group algebra $kE_r$, the constant $j$-rank property is in fact a property of the components of the stable Auslander-Reiten quiver of $kE_r$ \cite[4.7]{fp}. We will see, that the situation is rather different in our context.\\
Unlike $\EIP(2,r)$ and $\EKP(2,r)$, the category $\CJT(2,r)$ is neither closed under images nor under submodules and is hence more difficult to grasp categorically. However, $\CJT(2,r)$ is closed under direct summands \cite[3.7]{cafrpe08}. We are able to make more specific statements about the category $\CJT(2,r)=\CR^1(2,r)$ as opposed to $\CR^j(n,r)$ with $n>2$.

\begin{lemma}\label{cr}
Let $M \in \modi B(2,r)$ be regular and not isomorphic to an $X_{\alpha}$. Let
$$0 \rightarrow \tau(M) \rightarrow E \rightarrow M \rightarrow 0$$
be the Auslander-Reiten sequence ending in $M$. If two out of the three modules are of constant rank, then so is the third.
\end{lemma}
\begin{proof}
Since $ E \rightarrow M$ is right almost split and $X_{\alpha}$ is indecomposable, any morphism $X_{\alpha} \rightarrow M$ factors through $E$. Hence we get the following exact sequence 
$$0 \rightarrow \Hom_{\mathcal{K}_r}(X_{\alpha},\tau(M)) \rightarrow \Hom_{\mathcal{K}_r}(X_{\alpha},E) \rightarrow \Hom_{\mathcal{K}_r}(X_{\alpha},M) \rightarrow 0$$
and the assertion follows with Theorem \ref{main}.
\end{proof}

A direct consequence is the following

\begin{prop}\label{rk0}
Let $\mathcal{C}$ be a regular component in $\Gamma(\mathcal{K}_r)$.
\begin{enumerate}[(i)]
\item If all quasi-simple modules in $\mathcal{C}$ are of constant rank, then $\mathcal{C} \subseteq \CJT(2,r)$.
\item In particular, if $\mathcal{W}(\mathcal{C})=0$, then $\mathcal{C} \subseteq \CJT(2,r)$.
\end{enumerate}
\end{prop}

This especially tells us, that there are many indecomposable modules of constant Jordan type in Loewy length two that satisfy neither the equal images property nor the equal kernels property which is not the case if $r=2$. The inclusion $\ind \EIP(2,r) \cup \ind \EKP(2,r) \subseteq \ind \CJT(2,r)$ of indecomposable objects is proper if and only if $r>2$. Since $\mathfrak{F}_{(2,r)}$ is dense, this directly implies the same result for the categories $\EIP_2(kE_r),\ \EKP_2(kE_r)$ and $\CJT_2(kE_r)$.

\begin{prop}\label{rk1}
Let $\mathcal{C}$ be a regular component with $\mathcal{W}(\mathcal{C})=1$. Then either $\mathcal{C} \subseteq \CJT(2,r)$ or there are no indecomposable modules of constant rank in $\mathcal{C}$ apart from the modules in $\EIP(n,r) \cap  \mathcal{C}$ and $\EKP(n,r) \cap  \mathcal{C}$.
\end{prop}
\begin{proof}
We first of all show
\begin{enumerate}[(*)]\label{rk5}
\item Let $\mathcal{C}$ be a regular component with $\mathcal{W}(\mathcal{C})=n$ and let $W_{\mathcal{C}}$ and $M_{\mathcal{C}}$ be as in Theorem \ref{shape}. Then for all $1 \leq k \leq n$ we have the following:
If there exists $l\geq k$ such that
\begin{enumerate}[(i)]
\item $[l]\tau^{-k}(W_{\mathcal{C}})$ is of constant rank, then so is $[l']\tau^{-k}(W_{\mathcal{C}})$ for all $l' \geq k$.
\item $\tau^k(M_{\mathcal{C}})(l)$ is of constant rank, then so is $\tau^k(M_{\mathcal{C}})(l')$ for all $l' \geq k$.
\end{enumerate}
\end{enumerate}
{\it Proof of (*)}:
We show $(1)$, $(2)$ is dual. Let $l \geq k$ and $[l]\tau^{-k}(W_{\mathcal{C}})$ be of constant rank. Now assume that there is $l' >k$ minimal such that $[l']\tau^{-k}(W_{\mathcal{C}})$ does not have constant rank. The quasi-socle $\tau^{k-l'-1}(W_{\mathcal{C}})$ satisfies the equal images property and we have a short exact sequence (cf. \cite[2.2]{ri78})
$$0 \rightarrow \tau^{k-l'-1}(W_{\mathcal{C}})  \rightarrow [l']\tau^{-k}(W_{\mathcal{C}}) \rightarrow [l'-1]\tau^{-k}(W_{\mathcal{C}}) \rightarrow 0.$$ 
In view of Lemma \ref{cjt}, $[l'-1]\tau^{-k}(W_{\mathcal{C}})$ has constant rank, a contradiction to the choice of $l'$.\\

Now since $\mathcal{W}(\mathcal{C})=1$, we have $[l]\tau^{-k}(W_{\mathcal{C}})=\tau^{k}(M_{\mathcal{C}})(l)$ for all $k,l \in \mathbb{N}$ and furthermore
$$\mathcal{M}=\left\{[l]\tau^{-k}(W_{\mathcal{C}})| l \geq  k \geq 1 \right\}=\left\{M \in \mathcal{C}|M \notin \EIP(n,r)\cup \EKP(n,r)\right\}.$$
Now (*) implies that if the cone $\mathcal{M}$ contains an element of $\CJT(2,r)$, we have $\mathcal{M} \subseteq \CJT(2,r)$.
\end{proof}

{\bf Examples}

\begin{enumerate}
\item The component $\mathcal{C}_n$ containing $W_{n,2}^{(r)}$ for $n \geq 3$: It is $\mathcal{W}(\mathcal{C}_n)=0$ and hence Proposition \ref{rk0} implies that all modules in $\mathcal{C}_n$ have constant rank.
\item The component $\mathcal{C}_{\alpha}$ containing $X_{\alpha}$: We claim that there are no constant rank modules in $\mathcal{C}_{\alpha}$ apart from the equal images and equal kernels modules. In view of Statement (*) in the proof of Proposition \ref{rk1}, we only need to show that $[2]X_{\alpha}$ does not have constant rank. Following \cite[V]{cb}, it is
$$\Hom_{\mathcal{K}_r}(X_{\alpha},[j]X_{\alpha})=0$$
for all $j \geq 2$. Hence $\Hom_{\mathcal{K}_r}(X_{\alpha},[2]X_{\alpha})=0$. Since furthermore $[2]X_{\alpha} \notin \EKP(2,r)$, the module can't be of constant rank.
\item The component $\mathcal{C}_{\lambda}$ containing the module $E^{(\lambda)}$: Since $E^{(\lambda)}$ obviously does not have constant rank and  $\mathcal{W}(\mathcal{C}_{\lambda})=1$, Corollary \ref{rk1} implies that there are no modules of constant rank in $\mathcal{C}_{\lambda}$ apart from the equal kernels and equal images modules. 
\end{enumerate}

\section*{Acknowledgements}
The results of this paper are part of my doctoral thesis which I am currently writing at the University of Kiel. I would like to thank my advisor Rolf Farnsteiner for his support as well as helpful remarks and comments. Moreover, I would like to thank Julian K\"ulshammer for proof reading and helpful comments. In particular, I would like to thank Otto Kerner for pointing out literature and results on wild hereditary algebras.

\bibliography{lit}{}
\bibliographystyle{plain}
\end{document}